\documentclass[a4paper,12pt,reqno]{amsart}
\usepackage{graphicx}

\usepackage{amsmath,amstext,amssymb,amsopn,amsthm}
\usepackage{url}

\usepackage[margin=30mm]{geometry}
\usepackage{eucal,mathrsfs,dsfont}
\usepackage{soul}

\newtheorem{theorem}{Theorem}[section]
\newtheorem{corollary}[theorem]{Corollary}
\newtheorem{lemma}[theorem]{Lemma}
\newtheorem{proposition}[theorem]{Proposition}

\theoremstyle{definition}

\newtheorem{definition}[theorem]{Definition}

\theoremstyle{remark}

\numberwithin{equation}{section}

\newcommand{\eps}{\varepsilon}

\newcommand{\calA}{\mathcal{A}}

\newcommand{\calD}{\mathcal{D}}
\newcommand{\calW}{\mathcal{W}}

\newcommand{\R}{\mathds{R}}

\newcommand{\N}{{\mathds{N}}}

\newcommand{\RR}{\mathrm{I\kern-0.20emR}}
\newcommand{\D}{\mathrm{d}\kern0.2pt}
\newcommand{\Hs}{{H^{3/2}}}
\newcommand{\ve}{{\varepsilon}}
\newcommand{\vp}{{\varphi}}
\newcommand{\pd}[2][]{\frac{\partial #1}{\partial #2}}
\newcommand{\pdd}[2][]{\frac{\partial^2 #1}{\partial #2^2}}

\DeclareMathOperator{\relint}{Int{\mbox{$_{\{y=0\}}$}}}
\DeclareMathOperator{\relintyzero}{Int{\mbox{$_{\{y=y_0\}}$}}}

\DeclareMathOperator{\supp}{supp}
\DeclareMathOperator{\dist}{dist}

\newcommand{\ssf}{\Psi}

\title[On high spots of the fundamental sloshing eigenfunctions]{On high spots of the fundamental sloshing eigenfunctions in axially symmetric domains}
\author[T. Kulczycki]{Tadeusz Kulczycki}
\author[M. Kwa{\'s}nicki]{Mateusz Kwa{\'s}nicki}
\thanks{T. Kulczycki and M. Kwa{\'s}nicki were supported in part by MNiSW grant \# N N201 373136.}
\address{T. Kulczycki and M. Kwa{\'s}nicki \newline Institute of Mathematics, Polish Academy of Sciences, ul. Kopernika 18, 51-617 Wroc{\l}aw, Poland. Institute of Mathematics and Computer Science, Technical University of Wroc{\l}aw, Wybrzeze Wyspianskiego 27, 50-370 Wroc{\l}aw, Poland.}
\email{T.Kulczycki@impan.pl}
\email{M.Kwasnicki@impan.pl}

\begin{document}
\begin{abstract}
We investigate the classical eigenvalue problem that arises in hydrodynamics and is referred to as the sloshing problem. It describes free liquid oscillations in a liquid container $W \subset \R^3$. We study the case when $W$ is an axially symmetric, convex, bounded domain satisfying the John condition. The Cartesian coordinates $(x,y,z)$ are chosen so that the mean free surface of the liquid lies in $(x,z)$-plane and $y$-axis is directed upwards ($y$-axis is the axis of symmetry). Our first result states that the fundamental eigenvalue has multiplicity 2 and for each fundamental eigenfunction $\varphi$ there is a change of $x,z$ coordinates by a rotation around $y$-axis so that $\varphi$ is odd in $x$-variable.

The second result of the paper gives the following monotonicity property of the fundamental eigenfunction $\varphi$. If $\varphi$ is odd in $x$-variable then it is strictly monotonic in $x$-variable. This property has the following hydrodynamical meaning. If liquid oscillates freely with fundamental frequency according to $\varphi$ then the free surface elevation of liquid is increasing along each line parallel to $x$-axis during one period of time and decreasing during the other half period. The proof of the second result is based on the method developed by D. Jerison and N. Nadirashvili for the hot spots problem for Neumann Laplacian.  
\end{abstract}

\maketitle

\section{Introduction}

Linear water-wave theory is a widely-used approach that allows to determine the frequencies and modes of free oscillations of a liquid in a container. Such oscillations exist provided the liquid's upper surface is free and, in the framework of this theory, one obtains their frequencies and modes from a mixed Steklov problem. The latter involves a spectral parameter in the boundary condition on the free surface. This boundary value problem (usually referred to as the sloshing problem) has been the subject of a great number of studies over 250 years (see \cite{FK} for a historical review and \cite{BKPS2010}, \cite{GHLST2008}, \cite{I2005}, \cite{KK2001}, \cite{KK2004}, \cite{KKM2004} for some of recent literature). It is also worth pointing out here that other Steklov type eigenvalue problems have attracted considerable attention in last years. For some of these developments, see e.g. \cite{GP2010}, \cite{FS2009}, \cite{BK2004}, \cite{DS2003}.

Recently, the question of the so-called `high spots' defined by sloshing eigenfunctions corresponding to the fundamental eigenvalue attracted the authors' attention. This question is not only similar, but closely related to the long-standing `hot spots' conjecture of J. Rauch. (It is worth mentioning that a substantial progress has been achieved in studies of this conjecture for the Neumann Laplacian during the past decade; see, for example, the works \cite{S1999}, \cite{BB1999}, \cite{JN2000}, \cite{B2005}, \cite{B2006}.) Roughly speaking, the question about high spots concerns monotonicity properties of fundamental sloshing eigenfunctions (see subsection 1.3 for a detailed description). Several results about the location of high spots were proved in \cite{KK2009} and \cite{KK2010}. One of them deals with the fundamental eigenfunction (it is unique up to a non-zero factor) of the two-dimensional sloshing problem in the case when the domain's top interval is the one-to-one orthogonal projection of the bottom. The other one treat fundamental eigenfunctions in troughs (their cross-sections are subject to the same condition), and some vertical axisymmetric containers. Moreover, it was shown in \cite{KK2009} that for vertical-walled containers with horizontal bottom the question about high spots is equivalent to the hot spots conjecture. 

The aim of this paper is to study the location of high spots for fundamental eigenfunctions satisfying the three-dimensional sloshing problem in axially symmetric domains of rather general shape. It occurs that the method, which can be briefly characterised as the method of domain's deformation (it was developed by D. Jerison and N. Nadirashvili  in \cite{JN2000} in order to prove the hot spots conjecture for domains with two axes of symmetry) is adaptable for our purpose. The result demonstrating that eigenvalues and eigenfunctions of the sloshing problem depend continuously on the domain deformation are of interest in itself.

\subsection{Sloshing problem}

First we formulate the three-dimensional sloshing problem in its general form. 

Let an inviscid, incompressible, heavy liquid occupy a three-dimensional container bounded from above by a free surface, which in its mean position is a simply connected two-dimensional domain of finite diameter. Let Cartesian coordinates $(x,y,z)$ be chosen so that the mean free surface lies in the $(x,z)$-plane and the $y$-axis is directed upwards. The surface tension is neglected on the free surface, and we assume the liquid motion to be irrotational and of small amplitude. The latter assumption allows us to linearise boundary conditions on the free surface and this leads to the following boundary value problem for $\vp (x,y,z)$ --- the velocity potential of the flow with a time-harmonic factor removed:
\begin{align}
 & \Delta \vp = 0 && {\rm in}\  W, \label{lap} 
 \\ &  \pd[\vp]{y}  = \nu \vp && {\rm on}\  F, \label{nu} \\ & \frac{\partial \vp}{\partial n} = 0 && {\rm on}\  B. \label{nc}
\end{align}
Here $W \subset \{ (x,y,z) \in \R^3 : y < 0 \} $ is the domain which is supposed to be a bounded Lipschitz domain. The boundary $\partial W$  consists of a two-dimensional domain $F \subset \{(x, 0, z) : x, z \in \R\}$ referred to as the free surface (we assume that $F$ is a bounded Lipschitz domain) and $B = \partial W\setminus {F}$, 
the rigid container's bottom. In the whole paper we will refer to a domain $W$ with the above geometric properties as to a \emph{liquid domain}. Throughout the article, $\frac{\partial}{\partial n}$ denotes the normal derivative at $\partial W$, which is well-defined for almost every (with respect to the surface measure) point of $\partial W$. We understand that $\vp$ is a continuous function on $\overline{W}$, and that (\ref{nc}) is satisfied for all points on $B$ for which $\frac{\partial}{\partial n}$ is defined. We remark that in condition \eqref{nu}, the coefficient $\nu = \omega^2/g$ is the spectral parameter which involves the radian frequency $\omega$ of liquid's oscillations and the acceleration due to gravity $g$. 

The zero eigenvalue obviously exists for the problem \eqref{lap}--\eqref{nc}, but we exclude it with the help of the following orthogonality condition:
\begin{equation}
\int_{ F} \vp = 0 . \label{ort}
\end{equation}
It has been known since the 1950s that the problem \eqref{lap}--\eqref{ort} has a discrete spectrum; that is, there exists a sequence of eigenvalues
\begin{equation}
0 < \nu_1 \leq \nu_2 \leq \dots \leq \nu_n \leq \dots, \label{seq}
\end{equation}
each having a finite multiplicity equal to the number of repetitions in \eqref{seq}, and such that $\nu_n\to \infty$ as $n\to \infty$. These eigenvalues can be found by means of the variational principle (see, for example, \cite{M}), corresponding to the following Rayleigh quotient:
\begin{equation} 
\frac{\int_W \bigl[ (\pd[\vp]{x})^2 + (\pd[\vp]{y})^2 + (\pd[\vp]{z})^2  \bigr]} {\int_{F} \vp^2} , \label{rq}
\end{equation}
where $\vp$ is in the Sobolev space $H^1(W)$ and satisfies~\eqref{ort}. 
Thus $\nu_1$ is equal to the minimum of this quotient over the subspace of the Sobolev space $H^1 (W)$ which consists of functions satisfying~\eqref{ort}; the corresponding eigenfunction delivers the value $\nu_1$ to the quotient. The eigenfunctions $\vp_n$ ($n = 1,2,\dots$) belong to $H^1 (W)$ and form a complete system in an appropriate Hilbert space (see, for example, \cite{KK2001}).

\subsection{Axisymmetric containers}

\begin{figure}
\centering
\includegraphics[scale=0.8]{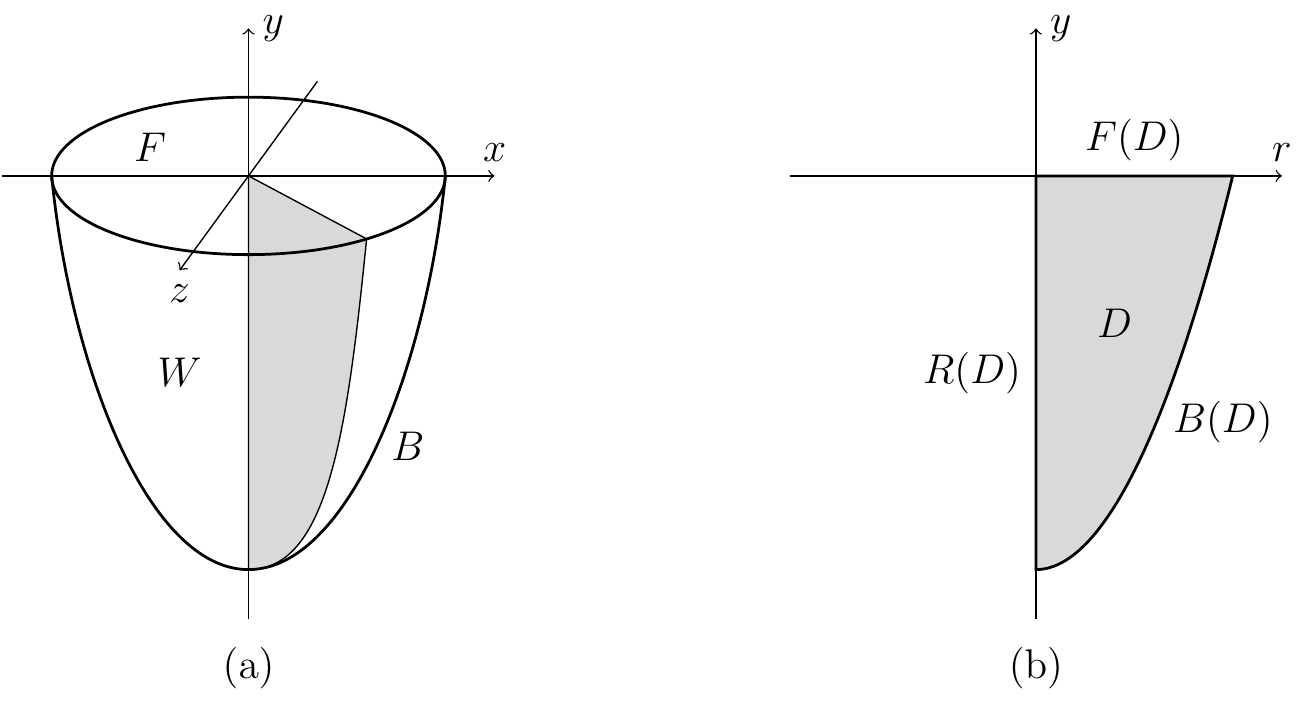}
\caption{(a) An axisymmetric container $W$ obtained as rotation of a domain $D$. (b) A domain $D$.}
\label{fig:1}
\end{figure}

Now we turn to the problem of sloshing in axisymmetric containers. It is convenient to introduce the cylindrical coordinates $(r,\theta,y)$ so that
\begin{equation}
x = r \cos \theta, \quad z = r \sin \theta . \label{cylindrical}
\end{equation}
and to take the $y$-axis as the axis of symmetry for $W$. In this case $F$ is typically a disc on $xz$-plane (see Figure~\ref{fig:1}(a)). Moreover, we will consider $W$ as obtained by rotation of a domain $D$ adjacent to both axes in the $ry$-plane. For such liquid domains, we write $W = W(D)$. It is convenient to think of $D$ as the cross-sections of $W$ along the half-plane $\theta = 0$.  By $F(D)$ and $B(D)$ we denote the cross-sections of $F$ and $B$, respectively,  while $R(D)$ is the part of $\partial D$ located on the $y$-axis (see Figure~\ref{fig:1}(b)).

It is clear that the ansatz
\begin{equation}
\vp = \psi(r,y) \cos (m \theta), \quad \text{or} \quad 
\vp = \psi(r,y) \sin (m \theta),
 \quad m = 0,1,2, \ldots , \label{rep}
\end{equation}
where $\psi$ is  bounded near $R(D)$, reduces the eigenvalue problem (\ref{lap})--(\ref{ort}) in $W$ to the following sequence of boundary value problems:
\begin{align}
\label{lapD}
& \pdd[\psi]{r} + \frac{1}{r} \,  \pd[\psi]{r}  + \pdd[\psi]{y} - \frac{m^2}{r^2} \, \psi = 0 && \text{in} \  D,\\
\label{nuD}
& \pd[\psi]{y} = \nu \psi && \text{on} \  F(D),\\
\label{ncD}
& \frac{\partial \psi}{\partial n} = 0 && \text{on} \  B(D) .
\end{align}
These relations must hold for all $m$, and  by~\eqref{ort}, for $m = 0$ we also have that
\begin{equation}
\label{ortD}
\int_{F(D)} \psi(r,y) r \, dr \, dy = 0 .
\end{equation}
The above reduction was applied by many authors, see, in particular, \cite[p.~56]{LBK1984}, \cite[p.~294]{KK2004}, \cite[formulas~(13), (14)]{GHLST2008}.

The variational method guarantees that for every $m = 0,1,2, \ldots$, the spectral problem (\ref{lapD})--(\ref{ortD}) has a sequence of eigenvalues 
\begin{equation}
0 < \nu_{m,1} \le \nu_{m,2} \le \nu_{m,3} \le \ldots , \quad m = 0,1,2, \dots , \label{meigenvalues}
\end{equation}
and by $\psi_{m,k}$, $m \ge 0$, $k \ge 1$, we denote the double sequence of corresponding eigenfunctions. Every eigenvalue in (\ref{meigenvalues}) has a finite multiplicity equal to the number of repetitions; moreover, for every $m \geq 0$ we have that $\nu_{m,k} \to \infty$ as $\ k \to \infty$. 

It is clear that, the sequence of eigenvalues $\{\nu_n\}_{n = 1}^{\infty}$ of problem (\ref{lap})--(\ref{ort}) for $W$ coincides with the double sequence $\{\nu_{m,k}\}_{m = 0, k = 1}^{\infty}$, with every number $\nu_{m,k}$ repeated twice when $m \ge 1$. Thus the sequence of problems (\ref{lapD})--(\ref{ortD}) is equivalent to the original sloshing problem. 

\subsection{Statement of results} We need the following definition. We say that a liquid domain $W$ satisfies \emph{the John condition} when $W \subset F \times (-\infty,0)$.

The following theorem is the main result of this paper (cf. \cite[Theorem~1.1]{JN2000}, \cite[Theorem 2.1]{KK2009}, \cite[Theorem 3.1]{KK2010}).
\begin{theorem}
\label{th:mainconvex}
Let us consider the sloshing problem (\ref{lap} - \ref{ort}). Assume that a liquid domain $W$ is an axisymmetric, convex, bounded domain, satisfying the John condition. We consider $W$ as obtained by rotation of a domain $D$ (see Figure~\ref{fig:1}). Then we have:
\begin{enumerate}
\item[(i)] The fundamental eigenvalue $\nu_1$ equals $\nu_{1,1}$ and has multiplicity 2. Two linearly independent eigenfunctions corresponding to $\nu_1$ are given in cylindrical coordinates by $\psi_{1,1}(r,y) \cos \theta$ and $\psi_{1,1}(r,y) \sin \theta$.
\item[(ii)] After multiplication by $\pm 1$ we may assume that $\psi_{1,1} > 0$ on $D$. We have 
\begin{equation*}
\frac{\partial \psi_{1,1}}{\partial r} > 0, \quad 
\frac{\partial \psi_{1,1}}{\partial y} > 0 \quad
\text{on} \quad  D.
\end{equation*}
\item[(iii)] Let $\varphi(x,y,z)$ be the eigenfunction corresponding to $\nu_1$ which is odd in $x$ variable (in cylidrical coordinates $\varphi$ equals $\psi_{1,1}(r,y) \cos \theta$). Denote $W_+ = \{(x,y,z) \in W: \, x > 0 \}$.
After multiplication by $\pm 1$ we may assume that $\vp > 0$ on $W_+$.  Then
\begin{equation*}
\frac{\partial \vp}{\partial x} > 0 \quad \text{on} \quad W, \quad 
\frac{\partial \vp}{\partial y} > 0 \quad \text{on} \quad W_+.
\end{equation*}
For any other eigenfunction $\varphi$ corresponding to $\nu_1$ there is a change of $x, z$ coordinates by a rotation around $y$-axis so that $\varphi$ is odd in $x$ variable.
\end{enumerate}
\end{theorem}

The third part of the theorem has the following hydrodynamical meaning. If liquid oscillates freely with the fundamental frequency $\nu_1$ then at every moment the free-surface elevation of liquid is proportional to the fundamental eigenfunction $\varphi(x,0,z)$ (see e.g. \cite{L1932}). If we assume that $\varphi(x,y,z)$ is odd in $x$ variable then the elevation is increasing along each line parallel to $x$ axis during one half-period of liquid oscillation and decreasing during the other half-period. In particular when the free surface is $F = \left\{ (x,y,z): \, x^2 + z^2 < r_0^2, \, y = 0\right\}$ then the elevation has its maximum at $(r_0,0,0)$ and minimum at $(- r_0,0,0)$ during one half-period of oscillation, whereas during the other half-period the maximum and minimum values exchange places with one another. This is the reason to call this property the `high spots' theorem. 

One could ask whether the assumption that $W$ satisfies the John condition is necessary in Theorem \ref{th:mainconvex}. It occurs that if $W$ does not satisfy the John condition then the monotonicity property of $\psi_{1,1}$ does not necessarily hold. More precisely we have

\begin{proposition}
\label{notJohn}
Let us consider the sloshing problem (\ref{lap} - \ref{ort}). Assume that $W$ is an axisymmetric liquid domain for which $F$ is a disk and $B$ is a $C^2$ surface. We consider $W$ as obtained by rotation of a domain $D$. After multiplication by $\pm 1$ we may assume that $\psi_{1,1} > 0$ on $D$. If the angle between $F(D)$ and $B(D)$ at the point where $F(D)$ and $B(D)$ meet (see Figure~\ref{fig:1}(b)) is bigger than $\pi/2$ and smaller than $\pi$ then $\psi_{1,1}$ attains maximum in the interior of $F(D)$, and $\frac{\partial \psi_{1,1}}{\partial r}$ changes the sign in $D$.
\end{proposition}

One could also ask whether assumption about convexity of $W$ is necessary in Theorem \ref{th:mainconvex}. Indeed we will show monotonicity property of $\psi_{1,1}$ for slightly more general class of domains, see Definition \ref{classD} and Theorem \ref{th:oddmon}. However for this class of domains we were not able to prove that $\nu_1 = \nu_{1,1}$.  

Although numerical results strongly suggest that Theorem \ref{th:mainconvex} (i) should hold the proof of Theorem \ref{th:mainconvex} (i) is far from being trivial. The most difficult part of this proof is to show that $\nu_1$ is not $\nu_{0,1}$ which is the smallest eigenvalue corresponding to an axially symmetric eigenfunction. The proof of Theorem \ref{th:mainconvex} (i) is based on results by Troesch \cite{T1960} obtained by inverse methods.

\subsection{Organization of the paper}

In Section 2 we prove Theorem \ref{th:mainconvex} (i). The rest of the paper deals with monotonicity properties of fundamental eigenfunctions. We use methods from \cite{JN2000}, which may be briefly described as deformations of domains. In Section 3 we first define a new class of domains $\calW$ and formulate monotonicity properties of $\psi_{1,1}$ for this class (see Theorem \ref{th:oddmon}). Then we prove continuos dependence of $\nu_{1,1}$ and $\psi_{1,1}$ under certain variations of the domain. In Section 4 we prove monotonicity properties of $\psi_{1,1}$ for a special class of piecewise smooth domains. In Section 5 we pass to the limit to obtain the same result for class $\calW$. As a conclusion we obtain Theorem \ref{th:mainconvex} (ii) and (iii). 

Throughout the article, except Definition \ref{classA}, Lemma \ref{covering} and Lemma \ref{H32}, we only study axisymmetric water domains $W = W(D)$, with free surface $F$ being a disk. We switch freely between Cartesian and cylindrical coordinate systems.  By scaling, it will be often sufficient to consider the case when $F$ is the unit disk $\{ (x, y, z) : x^2 + y^2 < 1, \, y = 0 \}$. By a standard argument, $\psi_{1,1}$ does not change sign in $D$, and $\nu_{1,1} < \nu_{1,2}$. With no loss of generality, we assume that $\psi_{1,1}$ is positive on $D$. We frequently use the continuity of $\psi_{1,1}$ on $\overline{D}$ and smoothness of $\psi_{1,1}$ in $D$.

\section{Fundamental eigenvalue}
In this section we prove Theorem \ref{th:mainconvex} (i). The most difficult element of the proof of this theorem is to exclude the possibility that $\nu_1$ equals $\nu_{0,1}$.

At first we need to introduce an auxiliary Dirichlet-Steklov problem (see \cite{BKPS2010} for formal introduction of this problem). 
\begin{align}
\label{eq:lap2}
 & \Delta \vp = 0 && \text{in $W$,} \\
\label{eq:nu2}
 & \pd[\vp]{y} = \tilde{\nu} \vp && \text{on $F$,} \\
\label{eq:nc2}
 & \vp = 0 && \text{on $B$.}
\end{align}
Here we assume that $W = W(D)$ is an axisymmetric liquid domain. The problem~\eqref{eq:lap2}--\eqref{eq:nc2} has a variational formulation similar to the one described in the Introduction for the sloshing problem (see~\eqref{rq}), with the only difference in the class of admissible functions for the Rayleigh quotient, which is now the space of $H^1(W)$ functions which vanish continuously at $B$ (see \cite{BKPS2010} for more details). Since $W$ is axisymmetric it is possible to use the same ansatz (\ref{rep}) for Dirichlet-Steklov problem as for the sloshing problem. We denote the eigenvalues of~\eqref{eq:lap2}--\eqref{eq:nc2} by $\tilde{\nu}_{m,k}(W)$ in a similar manner as for the sloshing problem~\eqref{lap}--\eqref{ort}. A standard argument shows that the first eigenvalue of~\eqref{eq:lap2}--\eqref{eq:nc2} is simple, it equals $\tilde{\nu}_{0,1}(W)$, the corresponding eigenfunction has constant sign, and it is the only eigenfunction with this property. 

Using standard arguments (see e.g. \cite[Section 3]{BKPS2010}, \cite{KK2009}, \cite{M}) one obtains the following domain monotonicity results for eigenvalues for both the Dirichlet-Steklov problem \eqref{eq:lap2}--\eqref{eq:nc2} and the sloshing problem \eqref{lap}--\eqref{ort}. We omit the proofs of these results because they are very similar to the proofs of Propositions 3.1.1 and 3.2.1 in \cite{BKPS2010}.

\begin{lemma}
\label{monotonicity1}
Let $W_1$, $W_2$ be axisymmetric liquid domains. If $W_1 \subset W_2$ and $F_1 \subset F_2$ then $\tilde{\nu}_{0,1}(W_1) \ge \tilde{\nu}_{0,1}(W_2)$.
\end{lemma}

\begin{lemma}
\label{monotonicity2}
Let $W_1$, $W_2$ be axisymmetric liquid domains. If $W_1 \subset W_2$ and $F_1 = F_2$ then ${\nu}_{0,1}(W_1) \le {\nu}_{0,1}(W_2)$ and ${\nu}_{1,1}(W_1) \le {\nu}_{1,1}(W_2)$.
\end{lemma}

In the rest of this section $W$ denotes a liquid domain satisfying assumptions of Theorem \ref{th:mainconvex} and such that $F$ is the unit disk. By scaling it is sufficient to consider only such case.

One of the important tools in the proof of Theorem \ref{th:mainconvex} (i) is the \emph{Stokes stream function} $\Psi$ corresponding to $\vp$ (see \cite[p.~125--127]{L1932}). Suppose that $\varphi$ is an axisymmetric eigenfunction of the sloshing problem \eqref{lap}--\eqref{ort}. Then $\Psi$ is the axisymmetric continuous function defined on $\overline{W}$ by
\begin{equation}
\label{eq:ssf}
  \frac{\partial \vp}{\partial r} = - \frac{1}{r} \frac{\partial \ssf}{\partial y} , \qquad \frac{\partial \vp}{\partial y} = \frac{1}{r} \frac{\partial \ssf}{\partial r}.
\end{equation}
Here we use cylindrical coordinates (\ref{cylindrical}). Since $\vp$ satisfies Neumann boundary condition on $B$, $\Psi$ is constant on $B$. Note that formula~\eqref{eq:ssf} defines $\ssf$ uniquely up to a constant. Hence we may assume that $\ssf = 0$ on $B$. Furthermore, $\ssf = 0$ when $r = 0$, and $\ssf$ satisfies in $W$ the relation
\begin{align*}
  \frac{\partial^2 \ssf}{\partial r^2} - \frac{1}{r} \frac{\partial \ssf}{\partial r} + \frac{\partial^2 \ssf}{\partial y^2} & = 0 , && r \ne 0 .
\end{align*}
In particular, $\ssf$ attains its maximum and minimum on $\{ r = 0 \} \cup \partial W$. From the boundary conditions it follows that the extreme values of $\ssf$ are attained on $\overline{F}$.

In this section we will need the following result 
which follows by a classical Courant-Hilbert argument (we omit the standard proof).

\begin{lemma}
Any axisymmetric eigenfunction of the problem \eqref{lap}--\eqref{ort} corresponding to $\nu_{0,1}(W)$ has two nodal domains. 
\end{lemma}

\begin{lemma}
\label{lem:ssf}
Suppose that  $\nu_{0,1}(W) < \tilde{\nu}_{0,1}(W)$. Then the Stokes stream function $\ssf$ corresponding to an axisymmetric eigenfunction $\vp$ with the eigenvalue $\nu_{0,1}(W)$ has constant sign on $W$.
\end{lemma}

\begin{proof}
Suppose, contrary to the hypothesis, that $\ssf$ does not have constant sign. By the maximum principle, $\ssf$ must change sign on $F$. Since in cylindrical coordinates $(r, \theta, y)$ we also have $\ssf(0,\theta,0) = \ssf(1,\theta,0) = 0$, it follows that there exist $0 < r_1 < r_2 < 1$ such that $\pd[\ssf]{r}(r_1, \theta, 0) = \pd[\ssf]{r}(r_2, \theta, 0) = 0$. But on $F$ we have $\nu_{0,1}(W) \vp = \pd[\vp]{y} = \frac{1}{r} \pd[\ssf]{r}$. It follows that the set $\{\vp = 0\}$ intersects $F$ along at least two circles $r = r_1$ and $r = r_2$. Since $\vp$ has only two nodal domains $W_1$, $W_2$, one of them, say $W_1$, must touch $F$ at the annulus $r_1 < r < r_2$. Since $W_2$ is connected, we conclude that $\partial W_1$ does not intersect $B$. Hence, $\vp$ restricted to $W_1$ is the first eigenfunction of the spectral problem~\eqref{eq:lap2}--\eqref{eq:nc2} in $W_1$ with $F$ and $B$ replaced by $\relint (F \cap \partial W_1)$ and $(\partial W_1) \setminus {F}$. Hence $\nu_{0,1}(W) = \tilde{\nu}_{0,1}(W_1)$. Here $\relint (F \cap \partial W_1)$ denotes the relative interior of the set $F \cap \partial W_1$ in the plane $y = 0$.

By domain monotonicity, we have $\nu_{0,1}(W) = \tilde{\nu}_{0,1}(W_1) \ge \tilde{\nu}_{0,1}(W)$, a contradiction with the assumption $\nu_{0,1}(W) < \tilde{\nu}_{0,1}(W)$.
\end{proof}


\begin{lemma}
\label{lem:nu01}
Suppose that  $\nu_{0,1}(W) < \tilde{\nu}_{0,1}(W)$. If the stream function $\ssf$ corresponding to an axisymmetric eigenfunction $\vp$ satisfies $\pd{r}(\frac{1}{r} \pd[\ssf]{r}) \le 0$ on $F$, then $\vp$ corresponds to the first axisymmetric eigenvalue $\nu_{0,1}(W)$.
\end{lemma}

\begin{proof}
Note that $\pd{r}(\frac{1}{r} \pd[\ssf]{r}) < 0$ at some point on $F$. Indeed, if $\pd{r}(\frac{1}{r} \pd[\ssf]{r})$ was identically zero on $F$, we would have $\ssf = C_1 + C_2 r^2$ on $F$, which contradicts $\ssf = 0$ for both $r = 0$ and $r = 1$.

Let $\nu$ be the eigenvalue corresponding to $\vp$. Let $\vp_{0,1}$ be an arbitrary axisymmetric eigenfunction corresponding to $\nu_{0,1}(W)$ and $\Psi_{0,1}$ be its stream function.
We have
\begin{equation*}
\int_F \vp \vp_{0,1} = \frac{1}{\nu \nu_{0,1}} \int_F \frac{\partial \vp}{\partial y} \frac{\partial \vp_{0,1}}{\partial y} = \frac{1}{\nu \nu_{0,1}} \int_F \frac{1}{r^2} \frac{\partial \ssf}{\partial r} \frac{\partial \ssf_{0,1}}{\partial r} .
\end{equation*}
Integration in polar coordinates and then integration by parts yield that
\begin{equation*}
\int_F \vp \vp_{0,1} = \frac{2 \pi}{\nu \nu_{0,1}} \int_0^1 \frac{1}{r} \frac{\partial \ssf}{\partial r} \frac{\partial \ssf_{0,1}}{\partial r}  \D r 
= -\frac{2 \pi}{\nu \nu_{0,1}} \int_0^1 \ssf_{0,1} \frac{\partial}{\partial r} \left(\frac{1}{r} \frac{\partial \ssf}{\partial r}\right)  \D r.
\end{equation*}
Since $\pd{r}(\frac{1}{r} \pd[\ssf]{r}) \le 0$ and it is not identically zero on $F$, and since $\ssf_{0,1}$ has constant sign on $F$ (by Lemma~\ref{lem:ssf}), we obtain that $\int_F \vp \vp_{0,1} \ne 0$. Hence $\vp$ corresponds to $\nu_{0,1}$.
\end{proof}

Now we come to the key element of the proof of Theorem \ref{th:mainconvex} (i), the application of the inverse method from Troesch paper \cite{T1960}.
In \cite[p.~283--284]{T1960} an eigenvalue $\nu$ corresponding to an axisymmetric eigenfunction for the sloshing problem is computed for a family of domains. With the notation of \cite{T1960}, we take $a_2 = 4$, so that the free surface is a unit disc. When $\lambda \in (0, 2]$, then for the domain
\begin{equation*}
 W_\lambda = \{ (x, y, z) : x^2 + z^2 < 4 y^2 + 8 y/\lambda + 1 , \, y < 0 \}
\end{equation*}
we simply have $\nu = \lambda$; the corresponding eigenfunction is the polynomial $\vp = 1 + \lambda y + 4 y^2 - 2 r^2 + (4/3) \lambda y^3 - 2 \lambda r^2 y$, where $r^2 = x^2 + z^2$, and the stream function corresponding to $\vp$ is $\ssf = (\lambda/2) (r^2 - r^4) + 4 r^2 y + 2 \lambda r^2 y^2$. Since $\pd{r}(\frac{1}{r} \pd[\ssf]{r}) \le 0$ on $F$, we may apply Lemma~\ref{lem:nu01} an we obtain the following corollary.

\begin{corollary}
\label{cor:Troesch}
For the set $W_\lambda$, either $\nu_{0,1}(W_\lambda) \ge \tilde{\nu}_{0,1}(W_\lambda)$ or the Troesch eigenvalue $\nu = \lambda$ is equal to $\nu_{0,1}(W_\lambda)$.
\end{corollary}

It is a natural conjecture that in fact always $\nu = \nu_{0,1}(W)$. However, we were not able to prove it. For our purposes the above corollary is enough.

In the proof of Theorem \ref{th:mainconvex} (i) we need some knowledge about sloshing eigenvalues for cylinders. These eigenvalues are well known (see e.g. \cite[example 2.1, p.~24]{BKPS2010}). We collect some results about these eigenvalues which we need in this section in the following lemma.

\begin{lemma}
\label{lem:eigencyl}
Let $h > 0$, $U_h  = \{(x,y,z): \, x^2 + z^2 < 1, \, -h < y < 0\}$, $F = \{(x,y,z): \, x^2 + z^2 < 1, \, y = 0\}$ and $B = \partial U_h \setminus {F}$. Let us consider the sloshing problem \eqref{lap}--\eqref{ort} in the cylinder $U_h$. Then we have
\begin{equation*}
\nu_{1,1}(U_h) = j'_{1,1} \tanh(j'_{1,1} h), \quad \quad 
\tilde{\nu}_{0,1}(U_h) = j_{0,1} \coth(j_{0,1} h),
\end{equation*}
where $j'_{1,1} \approx 1.8412$ is the first positive zero of $J'_{1}$ and $j_{0,1} \approx 2.4048$ is the first positive zero of $J_0$. Here $J_0$ and $J_1$ are Bessel functions of the first kind.

For any $h > 0$ we have $\nu_{1,1}(U_h) < j'_{1,1}$ and $\tilde{\nu}_{0,1}(U_h) > j_{0,1}$.
\end{lemma}

\begin{proof}[Proof of Theorem \ref{th:mainconvex} (i)]
Let us recall that we assume that $W$ is a liquid domain satisfying assumptions of Theorem 1.1 and $F = \{(x,y,z): x^2 + z^2 < 1, \, y = 0\}$.

First note that any eigenfunction of \eqref{lap}--\eqref{ort} corresponding to $\nu_1(W)$ must have 2 nodal domains. Eigenfunctions of \eqref{lap}--\eqref{ort} corresponding to $\nu_{m,k}(W)$, $m \ge 2$, of the shape $\psi_{m,k}(r,y) \cos(m \theta)$, $\psi_{m,k}(r,y) \sin(m \theta)$, have at least 4 nodal domains, so $\nu_{m,k}(W)$ cannot be equal to $\nu_1(W)$ for $m \ge 2$. Recall that $\nu_{1,2}(W) > \nu_{1,1}(W)$. Hence, in order to show that $\nu_1(W) = \nu_{1,1}(W)$, we only need to prove that $\nu_{0,1}(W) > \nu_{1,1}(W)$. 

Note that for $\lambda \in (0,2]$ we have $4 y^2 + 8 y/\lambda + 1 \le (1 + 4 y / \lambda)^2$, so that $W_\lambda$ is contained in the circular cone
\begin{equation*}
 V_\lambda = \{ (x, y, z) : x^2 + z^2 < (1 + 4 y / \lambda)^2 , \, y < 0 \} .
\end{equation*}
For $\lambda = 2$, in fact, $W_\lambda = V_\lambda$. The height of $V_\lambda$ is equal to $\lambda / 4$. Hence, $U_{\lambda/4}$ is the smallest vertical cylinder containing $V_\lambda$. By Lemma \ref{lem:eigencyl} $\nu_{1,1}(U_{\lambda/4}) = j'_{1,1} \tanh(j'_{1,1} \lambda / 4)$ and $\tilde{\nu}_{0,1}(U_{\lambda/4}) = j_{0,1} \coth(j_{0,1} \lambda / 4)$.

Let $h$ be the height of $W$. When $h \le 1/2$, then for $\lambda = 4 h$ we have $W_\lambda \subseteq V_\lambda \subseteq W \subseteq  U_{\lambda/4}$. There are two possibilities.
\begin{enumerate}
\item[(a)]
If $\nu_{0,1}(W_\lambda) \ge \tilde{\nu}_{0,1}(W_\lambda)$, then, by domain monotonicity and Lemma \ref{lem:eigencyl} we have $\nu_{0,1}(W_\lambda) \ge \tilde{\nu}_{0,1}(U_{\lambda/4}) > \nu_{1,1}(U_{\lambda/4})$.
\item[(b)]
Otherwise, by Corollary \ref{cor:Troesch} we have $\nu_{0,1}(W_\lambda) = \lambda$, and since $\lambda > (j'_{1,1})^2 \lambda / 4 > j'_{1,1} \tanh(j'_{1,1} \lambda/4) = \nu_{1,1}(U_{\lambda/4})$, we also have $\nu_{0,1}(W_\lambda) > \nu_{1,1}(U_{\lambda/4})$.
\end{enumerate}
Hence, in both cases, we have $\nu_{0,1}(W) \ge \nu_{0,1}(W_\lambda) > \nu_{1,1}(U_{\lambda/4}) \ge \nu_{1,1}(W)$, as desired.

When the height $h$ of $W$ is larger then $1/2$, then we simply have $W_2 = V_2 \subseteq W \subseteq U_h$. Again, there are two possibilities.
\begin{enumerate}
\item[(a)]
If $\nu_{0,1}(W_2) \ge \tilde{\nu}_{0,1}(W_2)$, then by domain monotonicity and  Lemma \ref{lem:eigencyl} we have $\nu_{0,1}(W_2) \ge \tilde{\nu}_{0,1}(W_2) \ge \tilde{\nu}_{0,1}(U_2) > j_{0,1} > j'_{1,1}$.
\item[(b)]
Otherwise, we have $\nu_{0,1}(W_2) = 2 > j'_{1,1}$.
\end{enumerate}
As before, in both cases we have, by domain monotonicity and Lemma \ref{lem:eigencyl}, $\nu_{0,1}(W) \ge \nu_{0,1}(W_2) > j'_{1,1} \ge \nu_{1,1}(U_h) \ge \nu_{1,1}(W)$, which completes the proof of the inequality $\nu_{0,1}(W) > \nu_{1,1}(W)$.

Hence $\nu_1(W) = \nu_{1,1}(W)$. There are exactly 2 linearly independent eigenfunctions corresponding to $\nu_1(W) = \nu_{1,1}(W)$: $\psi_{1,1}(r,y) \cos(\theta)$, $\psi_{1,1}(r,y) \sin(\theta)$. Hence $\nu_1(W)$ has multiplicity 2.
\end{proof}

\section{Continuous dependence under variation of the domain}

In this section we first define a new class of domains $\calW$ and formulate monotonicity properties of $\psi_{1,1}$ for this class. Then we prove continuos dependence of $\nu_{1,1}$ and $\psi_{1,1}$ under certain variations of the domain. Ideas used in this section are similar to the ideas from Section 2 in \cite{JN2000}.

Let us first describe the class $\calW$ in an informal way. The class $\calW$ consists of axially symmetric domains $W$ with horizontal cross-sections being circles or radius decreasing with depth $-y$. A similar condition for two-dimensional domains was assumed in Theorem~1.1 in \cite{JN2000}. For technical reasons we will assume certain regularity near the free surface and the vertical axis. More formally, we first describe the class $\calD$ of cross-sections.

\begin{definition}
\label{classD}
A domain $D \subset \{(r,y) : r > 0, \, y < 0\}$ belongs to the class of domains $\calD$ iff its boundary consists of the following 3 parts (see Figure~\ref{fig:2}):
\begin{enumerate}
\item[(i)] the horizontal interval $F(D) = \{(r,y): \, r \in [0,r_0), \, y = 0\}$, where $r_0 > 0$;
\item[(ii)] the vertical interval $R(D) = \{(r,y): \, r = 0, \, y \in (y_0,0]\}$, where $y_0 < 0$;
\item[(iii)] $B(D)$, parametrized by a simple continuous curve $(r(t), y(t))$, $t \in [0,T]$, satisfying the following conditions:
\begin{enumerate}
\item $(r(0), y(0)) = (0,y_0)$, $(r(T), y(T)) = (r_0,0)$, and $r(t) > 0$ and $y(t) < 0$ are nondecreasing for $t \in [0,T]$,
\item there exist $\varepsilon > 0$ ($\varepsilon \le T/2$) and $M \ge 1$ such that for $t \in [0,\varepsilon]$, $r(t) = t$ and $y(t)$ is a Lipschitz function with Lipschitz constant $M$, and for $t \in [T-\varepsilon,T]$, $y(t) = t - T$ and $r(t)$ is a Lipschitz function on $[T -\varepsilon,T]$ with Lipschitz constant $M$.
\end{enumerate}
\end{enumerate}
For fixed $\varepsilon > 0$, $M \ge 1$, $H > \ve$, $r_0 > \ve$, we write $D \in \calD(\varepsilon,M,H,r_0)$ when the above relations hold with the prescribed $\ve$, $M$ and $r_0$, and for some $y_0 \in (-H,-\ve)$. 
\end{definition}

\begin{figure}
\centering
\includegraphics[scale=0.8]{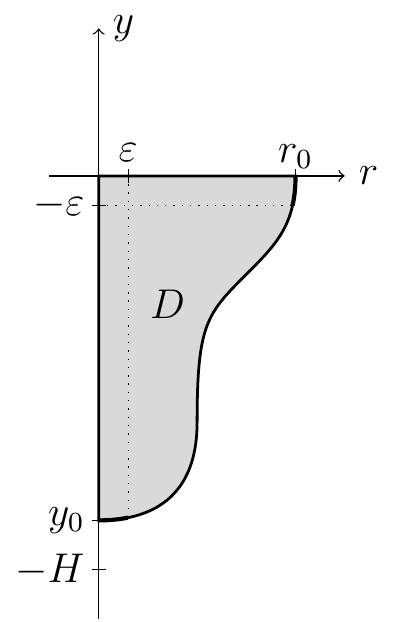}
\caption{An example of a domain $D$ belonging to the class $\calD$.}
\label{fig:2}
\end{figure}

\begin{definition}
The domain $W \subset \{(x,y,z): \, x,z \in \R, y < 0\}$ belongs to {\it{the class of domains $\calW$}} iff 
\[
W = W(D) = \{(x,y,z) \in \R^3: \, (\sqrt{x^2 + z^2},y) \in D \cup R(D)\}
\]
for some $D \in \calD$.

For $W = W(D) \in \calW$ we will always assume that its boundary $\partial{W}$ consists of 2 nonempty parts $F$, $B$ with $B = \partial{W} \setminus F$ and $F =  \{(x,y,z): x^2 + z^2 < r_0^2, \, \, y = 0\}$ where $r_0$ is the number appearing in the definition of $D$. 

By $\calW(\varepsilon,M,H,r_0)$ ($\varepsilon > 0$, $M \ge 1$, $H > \ve$, $r_0 > \ve$) we denote these domains $W = W(D)$ from $\calW$ for which $D \in \calD(\varepsilon,M,H,r_0)$.
\end{definition}

Note that all bounded, convex, axisymmetric domains satisfying the John condition belong to $\calW$. We are able to prove the monotonicity of $\psi_{1,1}$ stated in Theorem~\ref{th:mainconvex} for all $W \in \calW$.

\begin{theorem}
\label{th:oddmon}
Let $D \in \calD$, $W = W(D) \in \calW$. Let $\psi_{1,1}$ be the eigenfunction of \eqref{lapD}--\eqref{ortD} corresponding to the eigenvalue $\nu_{1,1}$ for the domain $W = W(D)$. After multiplication by $\pm 1$ we may assume that $\psi_{1,1} > 0$ on $D$. Then we have
\begin{align}
\label{eqmain1}
& \frac{\partial \psi_{1,1}}{\partial r} > 0 && \text{in} \  D, \\
\label{eqmain2}
& \frac{\partial \psi_{1,1}}{\partial y} > 0 && \text{in} \  D.
\end{align}
\end{theorem}

In order to use methods from \cite{JN2000} we need to introduce another class of domains. This class may be briefly desribed as star-shaped Lipschitz domains (cf. definition of class $L_M$ in \cite[p.~744]{JN2000}).
\begin{definition}
\label{classA}
Let $p_0 \in \R^3$. By $\calA(p_0)$ we denote the class of domains $W \subset \R^3$ such that 
\[
W = \left\{ p \in \R^3: \, 0 < |p - p_0| < f \left( \frac{p - p_0}{|p - p_0|}\right) \right\} 
\cup \left\{ p_0 \right\}
\]
for some positive Lipschitz function $f$ on the unit sphere $S^2 \subset \R^3$. If $0 < r_1 < r_2$ and $\eta > 0$, then by $\calA(p_0,r_1,r_2,\eta)$ we denote the subclass of $\calA(p_0)$ for which $f$ is a Lipschitz function with Lipschitz constant $\eta$ such that $r_1 < f < r_2$.
\end{definition}

\begin{lemma}
\label{covering}
Let $p_0 \in \R^3$. If $W \in \calA(p_0)$, then $W$ is a Lipschitz domain; that is, there exist constants $\delta > 0$, $N \in \N$, $L > 0$ such that the boundary $\partial{W}$ may be covered by balls $B_i$, $i = 1, \ldots, N$ of radii $\delta$, and such that for each $i$, $B_i \cap \partial{W}$ is the graph of a Lipschitz function with Lipschitz constant $L$. When $0 < r_1 < r_2$, $\eta > 0$ and $W \in \calA(p_0,r_1,r_2,\eta)$, then $\delta$, $N$ and $L$ depend only on $r_1, r_2, \eta$. 
\end{lemma}

The proof of this lemma is standard and is omitted. 

We will use the following notation: we will write $C(\alpha,\beta,\ldots)$ to indicate that $C$ is a constant depending only on $\alpha,\beta,\ldots$.

The following lemma is the crucial result which is taken from \cite{JN2000}.

\begin{lemma}
\label{H32}
Let $p_0 \in \R^3$ and let $W \in \calA(p_0)$ be a liquid domain contained in a ball $B(p_0, r_2)$. Let $\vp$ be a solution of the eigenvalue problem \eqref{lap}--\eqref{ort} such that $\int_F \vp^2 = 1$, and let $\nu$ be the corresponding eigenvalue. Then $\vp \in \Hs(W)$ and there is an extension $\tilde{\vp} \in \Hs(\R^3)$ of $\vp$ (that is, $\tilde{\vp} = \vp$ on $W$) such that $\supp \tilde{\vp} \subset B(x_0, r_2 + 1)$. In particular, $\nabla \tilde{\vp} \in L^3(\R^3)$.

Suppose that $0 < r_1 < r_2$, $\eta > 0$, $\nu_0 > 0$ and $a > 0$. If $W \in \calA(p_0, r_1, r_2, \eta)$, $\nu \le \nu_0$ and $\|\vp\|_{L^2(W)} \le a$, then $\|\vp\|_{\Hs(W)}$, $\|\tilde{\vp}\|_{\Hs(\R^3)}$ and $\| \nabla \tilde{\vp} \|_{L^3(\R^3)}$ are bounded above by constants depending only on $r_1$, $r_2$, $\eta$, $\nu_0$ and $a$.
\end{lemma}

\begin{proof}
The proof is based on the arguments used in the proof of Lemma 2.5 in \cite{JN2000}. For the convenience of the reader we write here much more details, but the main idea is exactly the same. As it will be seen below the fact that $\varphi \in H^{3/2}(W)$ follows easily from \cite{JK1982} and \cite{JK1995}. What is more difficult to justify is the fact that the norm $\|\varphi\|_{H^{3/2}(W)}$ depends only on $r_1$, $r_2$, $\eta$, $\nu_0$ and $a$, and the properties of the extension. In order to justify these facts we repeat some of the arguments from \cite{JK1995}.

Suppose that $W \in \calA(p_0, r_1, r_2, \eta)$. Consider $\vp$ as the solution of the Neumann problem
\[
\Delta \vp = 0 \quad \text{on} \quad W, \quad \frac{\partial \vp}{\partial n} = \nu 1_F \vp \quad \text{on} \quad \partial{W}.
\]
By $M(u)$ we denote the nontangential maximal function of the function $u: W \to \R$ defined for $q \in \partial W$ by
\[
M(u)(q) = \sup\{|u(p)|: \, p \in W, \, |p - q| < \kappa \dist(p,\partial W)\}
\]
(cf. \cite[p.~22]{JK1982}, or \cite{JK1981}). Here we take $\kappa = 2(L + 1)$ where $L$ is the constant from Lemma \ref{covering}. Of course $L$ and $\kappa$ depend only on $r_1$, $r_2$ and $\eta$. By $M(\nabla \vp)$ we mean $M(\frac{\partial \vp}{\partial x}) + M(\frac{\partial \vp}{\partial y}) + M(\frac{\partial \vp}{\partial z})$. Theorem 4.1 in \cite{JK1982} or Theorem 2 in \cite{JK1981} give that 
\begin{equation}
\label{eq:mnablavp}
\|M(\nabla \vp)\|_{L^2(\partial W)} \le C \|\nu 1_F \vp\|_{L^2(\partial W)} =
C \nu^2 \int_F \vp^2 = C \nu^2 \le C \nu_0^2 ,
\end{equation}
and Corollary 5.7  in  \cite{JK1995} (cf. also~\cite[Remark (b), p.~206]{JK1981}) gives that $\vp \in \Hs(W)$. 

Now our aim will be to show that the norm $\|\vp\|_{\Hs(W)}$ depends only on $r_1$, $r_2$, $\eta$, $\nu_0$, $a$. At first we want to justify that the norm $\|M(\nabla \vp)\|_{L^2(\partial W)}$ depends only on $r_1$, $r_2$, $\eta$ and $\nu_0$. This follows from the proof of Theorem 2  in  \cite{JK1981}. Indeed this theorem is proved first for star-shaped smooth domains $\Omega = \{(r,\theta): \, 0 \le r < \Psi(\theta)\}$, $\Psi \in C^{\infty}(S^{n - 1})$, see  \cite[p.~205, line 4]{JK1981}. For such domains the assertion of Theorem 2 holds with a constant depending only on the Lipschitz constant  of the function $\Psi$, $\|\Psi\|_{\infty}$, $\min_{\theta \in S^{n - 1}} |\Psi(\theta)|$  and $\kappa$,  see  \cite[p.~205, formula~(5)]{JK1981}. The rest of the proof of Theorem 2 is the approximation of general star-shaped domains by smooth star-shaped domains. It occurs that the constant $C$ in the formulation of Theorem 2 depends only on the Lipschitz constant  of the function $\Psi$ defing the star-shaped Lipschitz domain, $\|\Psi\|_{\infty}$, $\min_{\theta \in S^{n - 1}} |\Psi(\theta)|$  and $\kappa$. In our situation this gives that the constant $C$ appearing in (\ref{eq:mnablavp}) depends only on $r_1$, $r_2$ and $\eta$. Note that in \cite{JK1981} the nontangential maximal function is defined for $\kappa = 2$ but the proof for the constant $\kappa = 2(L + 1)$ is exactly the same.

Now let 
$$
\Gamma(q) = \{p \in W: \, |p - q| < \kappa \dist(p,\partial W)\}, \quad q \in \partial W
$$
and define the \emph{area integral} of a function $v: W \to \R$ by
\[
S(v)(q) = \left(\int_{\Gamma(q)} \dist(p,\partial W)^{-1} |\nabla v(p)|^2 \, dp \right)^{1/2}.
\]
By $S(\nabla v)$ we mean $S(\frac{\partial v}{\partial x}) + S(\frac{\partial v}{\partial y}) + S(\frac{\partial v}{\partial z})$. By Theorem 5.4 in \cite{JK1995} (cf. also  \cite[Corollary 5.7, (a)$\implies$(c)]{JK1995})  and~\eqref{eq:mnablavp} we obtain
\[
\|S(\nabla \vp)\|_{L^2(\partial W)} \le C \|M(\nabla \vp)\|_{L^2(\partial W)} \le C',
\]
where $C' = C'(r_1,r_2,\eta,\nu_0)$. We claim that (cf. \cite[Corollary~5.7]{JK1995})
\begin{equation}
\label{nabla22}
\int_W \dist(p,\partial W) |\nabla^2\vp(p)|^2 \, dp < C'',
\end{equation}
where $C'' = C''(r_1,r_2,\eta,\nu_0)$. Here $\nabla^2 \vp$ is the vector of all  second derivatives of $\vp$ (cf. \cite[p.~181, line~6]{JK1995}). The inequality for the integral over a neighborhood of $\partial W$ follows from the estimate of $\|S(\nabla \vp)\|_{L^2(\partial W)}$ by appling the argument used in the proof of the upper bound inequality in (6.1) of \cite{D1980} to each of $\frac{\partial \vp}{\partial x}$, $\frac{\partial \vp}{\partial y}$, $\frac{\partial \vp}{\partial z}$. The integral over  the remaining compact subset of $W$ is bounded using a simple gradient estimate for each of the harmonic functions $\frac{\partial \vp}{\partial x}$, $\frac{\partial \vp}{\partial y}$, $\frac{\partial \vp}{\partial z}$, and the inequality $\int_W |\nabla \vp|^2 = \nu \le \nu_0$. Our claim is proved. 

Now we will argue like in Corollary 5.7, (c)$\implies$(b) in  \cite{JK1995}. Note that $\nabla \vp$ is harmonic so one could use Corollary 5.5, (c)$\implies$(b) in  \cite{JK1995} for $v = \nabla \vp$. The implication  (c)$\implies$(b) in Corollary~5.5 in \cite{JK1995}  follows from the proof of Theorem 4.1, (b)$\implies$(a) in  \cite{JK1995} for $u = v = \nabla \vp$. Indeed we have 
$$
\|\delta^{1/2}(p) \nabla u(p)\|_{L^2(W)} + \|u\|_{L^2(W)} \le C(r,R,\eta,\nu_0)
$$
for $u = \nabla \vp$, $\delta(p) = \dist(p,\partial{W})$. From the proof of Theorem 4.1, (b)$\implies$(a) in  \cite{JK1995} it follows that
$$
u = \nabla \vp \in [L^2(W),H^1(W)]_{1/2,2}
$$
and $\|\nabla \vp\|_{[L^2(W),H^1(W)]_{1/2,2}} \le C(r_1,r_2,\eta,\nu_0)$. The last inequality follows from the proof of Theorem 4.1, (b)$\implies$(a) in  \cite{JK1995} and Lemma \ref{covering}.  Here  $[L^2(W),H^1(W)]_{1/2,2}$ is the interpolation space of power $1/2$ given by the real interpolation method  and  $\nabla \vp \in [L^2(W),H^1(W)]_{1/2,2}$ means that  $\frac{\partial \vp}{\partial x}, \frac{\partial \vp}{\partial y}, \frac{\partial \vp}{\partial z} \in [L^2(W),H^1(W)]_{1/2,2}$. 

Now we will argue in the similar way like in Proposition 2.17  in  \cite{JK1995} (or in Proposition 2.4  in  \cite{JK1995} with real interpolation instead of complex interpolation). By Theorem  VI.5  in \cite{S1970} for any bounded Lipschitz domain $D \subset \R^n$ there is an extension operator $E$ mapping functions on $D$ to functions on $\R^n$ such that $Ef(x) = f(x)$, $x \in D$, and such that for each $k \in \N$, $E$ is a bounded operator from $H^k(D)$ to $H^k(\R^n)$ (cf. also Theorem 2.3 in \cite{JK1995}) ($H^0 = L^2$). When $D$ is contained in a ball $B(p,R)$, $R > 0$, the operator $E$ may be chosen so that for any function $f$ on $D$ we have $\supp(Ef) \subset B(p,R + 1)$. Assume that there exist numbers $\delta > 0$, $N \in \N$, $L > 0$ such that the boundary $\partial D$ may be covered by balls $B_i$, $i = 1, \ldots, N$ of radii $\delta$ such that for each $i$, $B_i \cap \partial D$ is the graph of a Lipschitz function with Lipschitz constant $L$. It follows from Theorem  VI.5' in  \cite{S1970} and the proof of Theorem  VI.5 in  \cite{S1970} ( pages 190--192,  see in particular (31)  on  page 191 and inequalities on page 192) that the norm of $E: H^k(D) \to H^k(\R^n)$ depends only on $n$, $k$, $\delta$, $N$, $L$. What is more formula (1.19), page 174 from \cite{JK1995}, arguments after this formula and formula (31), page 191  in  \cite{S1970} give that 
\begin{equation}
\label{Enabla}
\frac{\partial}{\partial x_j}(E f) = 
\sum_{i = 1}^n Q_{j,i}(\frac{\partial}{\partial x_j} f) + S_j(f),
\quad j = 1, \ldots, n,
\end{equation}
where $Q_{j,i}$, $S_j$ are bounded linear operators from $H^k(D)$ to $H^k(\R^n)$ such that their norms depend only on $n$, $k$, $\delta$, $N$, $L$. It follows that if $f, \nabla f \in H^k(D)$ then $\nabla(E f) \in H^k(\R^n)$ and $\|\nabla(E f)\|_{H^k(\R^n)} \le C(n,k,\delta,N,L) (\|\nabla f\|_{H^k(D)} + \|f\|_{H^k(D)})$

Now let us come back to our situation. By Lemma \ref{covering} the norms of $E$, $Q_{j,i}$, $S_j$ as bounded linear operators from $H^k(W)$ to $H^k(\R^3)$ depend only on $k$, $r_1$, $r_2$, $\eta$. We know that $\vp \in H^1(W)$, $\nabla \vp \in [H^0(W),H^1(W)]_{1/2,2}$ and $\|\vp\|_{H^1(W)} \le C(r_1,r_2,\eta,\nu_0,a)$, $\|\nabla \vp\|_{[H^0(W),H^1(W)]_{1/2,2}} \le C(r_1,r_2,\eta,\nu_0)$. By the real interpolation method and (\ref{Enabla}) (cf. the proof of Proposition 2.4  in  \cite{JK1995}) it follows that $E \vp \in H^1(\R^3)$, $\nabla(E \vp) \in [H^0(\R^3),H^1(\R^3)]_{1/2,2}$ and also $\|E \vp\|_{H^1(\R^3)} \le C(r_1,r_2,\eta,\nu_0,a)$, $\|\nabla(E \vp)\|_{[H^0(\R^3),H^1(\R^3)]_{1/2,2}} \le C(r_1,r_2,\eta,\nu_0,a)$.

We have
$$
[H^0(\R^3),H^1(\R^3)]_{1/2,2} = B_{2,2}^{1/2}(\R^3) \subset H^{1/2}(\R^3) \subset L^3(\R^3),
$$
where $B_{2,2}^{1/2}(\R^3)$ is the Besov space and both inclusions are continuous embeddings. The equality follows from Theorem 6.2.4 \cite{BL1976}, the first inclusion follows from Theorem 6.4.4  in  \cite{BL1976} and the second inclusion follows from Theorem 6.5.1  in  \cite{BL1976}. 

Recall also that $\supp(E \vp), \supp(\nabla(E \vp)) \subset B(p_0,r_2 + 1)$. It follows that $\nabla(E \vp) \in L^3(B(p_0,r_2 + 1))$ and $\|\nabla(E \vp)\|_{L^3(B(p_0,r_2 + 1))} \le C(r_1,r_2,\eta,\nu_0,a)$. We also have $E \vp \in H^1(B(p_0,r_2 + 1))$, $\nabla(E \vp) \in H^{1/2}(B(p_0,r_2 + 1))$,  and the norms  $\|E \vp\|_{H^1(B(p_0,r_2 + 1))}$, $\|\nabla(E \vp)\|_{H^{1/2}(B(p_0,r_2 + 1))}$  are bounded by  $C(r_1,r_2,\eta,\nu_0,a)$. It follows that ( see e.g. \cite[formula~(38)]{ES1992},  cf. also \cite[Proposition 2.18(a)]{JK1995}) $E \vp \in H^{3/2}(B(p_0,r_2 + 1))$ and $\|E \vp\|_{H^{3/2}(B(p_0,r_2 + 1))} \le C(r_1,r_2,\eta,\nu_0,a)$.
\end{proof}

Recall that by scaling it is sufficient to consider the domains $W$ such that $F$ is a unit disk. For this reason we only consider $W \in \calW(\varepsilon,M,H,1)$ (that is we fix $r_0 = 1$).

\begin{figure}
\centering
\includegraphics[scale=0.8]{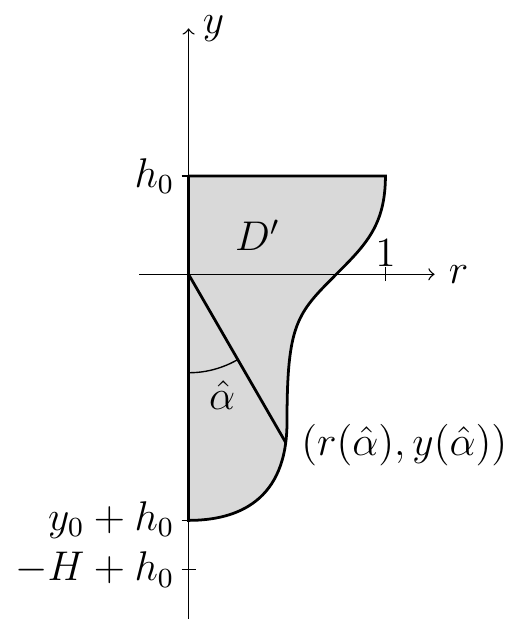}
\caption{An auxiliary picture for Lemma \ref{angle}.}
\label{fig:3}
\end{figure}

\begin{lemma}
\label{angle}
Fix $\ve \in (0, 1)$, $M \ge 1$, $H > \ve$. Assume that $D \in \calD(\ve,M,H,1)$. Let $h_0 = \ve/(2 M)$ and let $D'$ be the translation of $D$ by the vector $(0, h_0)$ (see Figure \ref{fig:3}). Let $(\hat{r}, \hat{\alpha})$ be the polar coordinate system in the $(r, y)$ plane, with $y = \hat{r} \cos \hat{\alpha}$ and $r = \hat{r} \sin \hat{\alpha}$. Then there is a Lipschitz function $f$ on $[0, \pi]$ such that $D' = \{ 0 < \hat{r} < f(\hat{\alpha}) , \, 0 < \hat{\alpha} < \pi \}$. Furthermore, the Lipschitz constant of $f$, the infimum of $f$ and the supremum of $f$ depend only on $\ve, M, H$.
\end{lemma}

\begin{proof}
The proof of this lemma is standard. Let $(r(t), y(t))$, $t \in [0, T]$, be the curve from the definition of $\calD$. Then the boundary of $D'$ consists of a part of the $y$-axis, a horizontal interval, and the curve $(r(t), y(t) + h_0)$, $t \in [0, T]$. In the polar coordinates, this curve is given by
\[
 \hat{r}(t) = \sqrt{(r(t))^2 + (y(t) + h_0)^2} , \quad \hat{\alpha}(t) = \frac{\pi}{2} + \arctan \frac{y(t) + h_0}{r(t)} \, .
\]
Our goal is to prove that $\hat{\alpha}(t)$ is increasing, and that $f = r \circ \hat{\alpha}^{-1}$ is a Lipschitz function.

By an appropriate reparametrization, with no loss of generality we assume that $r(t)$ and $y(t)$ are absolutely continuous functions of $t$, and $r'(t) + y'(t) > 0$ for a.e. $t \in [0, T]$. Since $r(t) \ge r(\ve) = \ve$ for $t \in [\ve, T]$, and $y(t) \le y(T - \ve) = -\ve$ for $t \in [0, T - \ve]$, we have $\hat{r}(t) \ge \ve - h_0$ for all $t$. Furthermore, $r'(t) = 1$ and $0 \le y'(t) \le M$ for a.e. $t \in [0, \ve]$, and similarly $0 \le r'(t) \le M$ and $y'(t) = 1$ for a.e. $t \in [T-\ve, T]$. We find that $\hat{r}(t)$ and $\hat{\alpha(t)}$ are absolutely continuous, and for a.e. $t \in [0,T]$,
\begin{align*}
 \hat{r}'(t) & = \frac{r(t) r'(t) + (y(t) + h_0) y'(t)}{\hat{r}(t)} \, , &
 \hat{\alpha}'(t) & = \frac{r(t) y'(t) - (y(t) + h_0) r'(t)}{(\hat{r}(t))^2} \, .
\end{align*}
For a.e. $t \in [\ve, T-\ve]$, we have
\[
 \hat{\alpha}'(t) \ge \frac{\ve y'(t) + (\ve - h_0) r'(t)}{(\hat{r}(t))^2} \ge \frac{\ve}{2} \, \frac{y'(t) + r'(t)}{(\hat{r}(t))^2} \, . 
\]
For a.e. $t \in [0, \ve]$,
\[
 \hat{\alpha}'(t) = \frac{t y'(t) - (y(t) + h_0)}{(\hat{r}(t))^2} \ge \frac{\ve - h_0}{(\hat{r}(t))^2} \ge \frac{\ve}{2(1 + M)} \frac{y'(t) + r'(t)}{(\hat{r}(t))^2} \, .
\]
Finally, for a.e. $t \in [T-\ve, T]$,
\[
 \hat{\alpha}'(t) = \frac{r(t) - (t - T + h_0) r'(t)}{(\hat{r}(t))^2} \ge \frac{\ve - h_0 M}{(\hat{r}(t))^2} \ge \frac{\ve}{2} \frac{1}{(\hat{r}(t))^2} \ge \frac{\ve}{2(1 + M)} \frac{y'(t) + r'(t)}{(\hat{r}(t))^2} \, . 
\]
In particular, $\hat{\alpha}(t)$ is strictly increasing, with $\hat{\alpha}(0) = 0$ and $\hat{\alpha}(T) = \pi/2 + \arctan(h_0)$. Furthermore, in a similar manner, we have for a.e. $t \in [0, T]$,
\[
 |\hat{r}'(t)| \le \frac{r(t) r'(t) + |y(t) + h_0| y'(t)}{\hat{r}(t)} \le (1 + H + \ve) \, \frac{r'(t) + y'(t)}{\hat{r}(t)} \, . 
\]
Hence, for a.e. $t \in [0, T]$,
\[
 \frac{|\hat{r}'(t)|}{\hat{\alpha}'(t)} \le (1 + H + \ve) \, \frac{2 (1 + M)}{\ve} \, \hat{r}(t) \le \frac{2 (1 + M) (1 + H + \ve)^2}{\ve} \, .
\]
It follows that $D' = \{ 0 < \hat{r} < f(\hat{\alpha}) , \, 0 < \hat{\alpha} < \pi \}$ with $f(\hat{\alpha})$ given by $f(\hat{\alpha}(t)) = \hat{r}(t)$, $t \in [0, T]$ (that is, for $\hat{\alpha} \in [0, \hat{\alpha}(T)]$), and by $f(\hat{\alpha}) = -h_0 / \cos \hat{\alpha}$ for $\hat{\alpha} \in [\hat{\alpha}(T), \pi]$. Since $\hat{r}(t) \ge \ve - h_0 \ge \ve/2$, $f$ is bounded below by $\ve/2$. Also, $f(\hat{\alpha}) \le 1 + H + \ve$. Finally, $f$ is absolutely continuous, for a.e. $t \in [0, T]$ we have
\[
 |f'(\hat{\alpha}(t))| = \frac{|\hat{r}'(t)|}{\hat{\alpha}'(t)} \le \frac{2 (1 + M) (1 + H + \ve)^2}{\ve} \/ ,
\]
and for $\hat{\alpha} \ge \hat{\alpha}(T)$,
\[
 |f'(\hat{\alpha})| = \frac{h_0 \sin \hat{\alpha}}{(\cos \hat{\alpha})^2} \le \frac{h_0}{(\sin(\arctan(h_0)))^2} = \frac{h_0^2 + 1}{h_0} \/ ,
\]
so that $f$ is a Lipschitz function with Lipschitz constant depending on $\ve, M, H$.
\end{proof}

As an immediate conclusion of this lemma we obtain the following result.

\begin{corollary}
\label{inclusion}
Let $\ve \in (0, 1)$, $M \ge 1$, $H > \ve$ and assume that $W \in \calW(\ve,M,H,1)$. Then there exist $p_0 = (0,- \eps/(2M),0)$ and $0 < r_1 < r_2 $, $\eta \ge 1$ depending only on $\ve$, $M$, $H$ such that $W \in \calA(p_0,r_1,r_2,\eta)$.
\end{corollary}

As in \cite{JN2000} we define the distance between domains (cf. page 745 in \cite{JN2000}). 

\begin{definition}
\label{distance} Fix $\eps \in (0, 1)$, $M \ge 1$, $H > \eps$. Put $p_0 = (0,- \eps/(2M),0)$. Let $W_1, W_2 \in \calW(\eps,M,H,1)$. Corollary \ref{inclusion} gives that  $W_i \in \calA(p_0)$, so
$$
W_i = \left\{ p \in \R^3: \, 0 < |p - p_0|  < f_i \left( \frac{p - p_0}{|p - p_0|}\right) \right\} 
\cup \left\{ p_0 \right\} , \quad i = 1, 2 ,
$$
for some Lipschitz functions $f_i$ on the unit sphere $S^2$.
We define the distance between $W_1, W_2 \in \calW(\eps,M,H,1)$ by
$$
d_{\calW(\eps,M,H,1)}(W_1,W_2) = \|f_1 - f_2\|_{\infty}.
$$
When it is clear to which class $W_1, W_2$ belong we will abbreviate $d_{\calW(\eps,M,H,1)}(W_1,W_2)$ to $d(W_1,W_2)$.
\end{definition}
Roughly speaking, we measure the distance between $W_1$, $W_2 \in \calW(\eps,M,H,1)$ as the natural distance between star-shaped domains in the class $\calA(p_0)$ for a special choice of $p_0$ depending on $\eps$, $M$.

\begin{lemma}
\label{lem:uniformapprox1}
Assume that $D_s \in \calD(\eps,M,H,1)$ for $s$ in a neighborhood of $0$, for some fixed $\eps \in (0, 1)$, $M \ge 1$, $H > \eps$. Let $r_s(t)$, $y_s(t)$ be functions defining $B(D_s)$ in the sense of Definition~\ref{classD}. Assume that all functions $r_s$, $y_s$ are defined on $[0,T]$, and that $\|r_s - r_0\|_{\infty} \to 0$ and $\|y_s - y_0\|_{\infty} \to 0$ as $s \to 0$. Then
$$
 d(W_s,W_0) = d_{\calW(\eps,M,H,1)}(W_s,W_0) \to 0, \quad \text{as} \quad s \to 0,
$$
where $W_s = W(D_s)$.
\end{lemma}

\begin{proof}
Let $\hat{r}_s(t)$, $\hat{\alpha}_s(t)$, $f_s(\hat{\alpha})$ be defined as in Lemma~\ref{angle} (see Figure~\ref{fig:3}), but for the domain $D_s$. Note that $f_s(\hat{\alpha}) = f_0(\hat{\alpha})$ for $\hat{\alpha}$ greater than $\hat{\alpha}_s(T) = \hat{\alpha}_0(T)$.

Clearly, $\hat{r}_s(t) \to \hat{r}_0(t)$ as $s \to 0$ uniformly in $t \in [0, T]$. In a similar manner, since $r_s(t) \ge \eps$ for $t \in [\eps, T]$, we have $\hat{\alpha}_s(t) \to \hat{\alpha}_0(t)$ uniformly in $t \in [\eps, T]$. Furthermore, for $t \in [0, \eps]$ we have $y_s(t) + h_0 \le -(\eps - h_0) \le -\eps/2$, and $\hat{\alpha}_s(t) = -\arctan(r_s(t) / (y_s(t) + h_0))$. It follows that $\hat{\alpha}_s(t) \to \hat{\alpha}_0(t)$ uniformly also in $t \in [0, \eps]$. Finally, since all of $f_s$ are Lipschitz functions with Lipschitz constant $\eta$, we have
\begin{align*}
 |f_s(\hat{\alpha}_0(t)) - f_0(\hat{\alpha}_0(t))| & \le |f_s(\hat{\alpha}_0(t)) - f_s(\hat{\alpha}_s(t))| + |f_s(\hat{\alpha}_s(t)) - f_0(\hat{\alpha}_0(t))| \\ & \le \eta |\hat{\alpha}_0(t) - \hat{\alpha}_s(t)| + |\hat{r}_s(t) - \hat{r}_0(t)| ,
\end{align*}
which converges to $0$ uniformly in $t \in [0, T]$. Therefore, $f_s(\hat{\alpha}) \to f_0(\hat{\alpha})$ uniformly in $\hat{\alpha}$, and the lemma follows.
\end{proof}

From now on we will be interested in eigenfunctions of \eqref{lap}--\eqref{ort} corresponding to $\nu_{1,1}$. In the rest of the paper we will use the following notation:
\begin{equation}
\label{psi11}
 \psi(r,y) = \psi_{1,1}(r,y),  \qquad \vp(r,\theta,y) = \psi_{1,1}(r,y) \, \cos \theta , \qquad \nu = \nu_{1,1} . 
\end{equation}
We use here cylindrical coordinates $(r,\theta,y)$ defined in~\eqref{cylindrical}, cf.~\eqref{rep}. By $\varphi(x,y,z)$  we denote the function $\varphi(r,\theta,y)$ written in Cartesian coordinates $(x,y,z)$. Recall that $\varphi(x,y,z)$ is one of the eigenfunctions of \eqref{lap}--\eqref{ort} corresponding to $\nu = \nu_{1,1}$, $\varphi(-x,y,z) = -\varphi(x,y,z)$, and $\psi > 0$ on $D$. We will always assume that $\varphi$ is normalized so that $\int_F \varphi^2 = 1$. Since in this section we discuss the problem~\eqref{lap}--\eqref{ort} for more than one liquid domain, we will indicate the dependence on $W$ in the subscript, as in $\psi_W$, $\varphi_W$, $\nu_W$.

\begin{lemma}
\label{uniform}
Let $W \in \calW(\varepsilon,M,H,1)$. There exist absolute constants $C_1$, $C_2$ such that
$$
\nu_W = \int_W |\nabla \vp_W|^2 \le C_1,
\quad \quad
\int_W \vp^2_W \le C_2.
$$
\end{lemma}
\begin{proof}
Note that $W$ is a subset of a cylinder $\{(x,y,z): \, x^2 + z^2 < 1, \, -H < y < 0\}$. By Lemmas \ref{monotonicity2} and \ref{lem:eigencyl} we get $\nu_W < j'_{1,1}$. By the variational characterization of $\nu_W$ and $\int_F \vp_W^2 = 1$ we have $\nu_W = \int_W |\nabla \vp_W|^2$.

Let $W_x$ be the orthogonal projection of $W$ on the $yz$-plane. For any $p \in W_x$, let $l(p)$ be the cross-section of $W$ with the line parallel to the $x$-axis and passing through $p$. Then $l(p)$ is an interval symmetric with respect to the $yz$-plane, and $l(p)$ has length $|l(p)| \le 2$. Since $\vp_W(x,y,z)$ is odd with respect to the $x$-variable, for any $p \in W_x$ we have
$$
\int_{l(p)} \vp_W^2
\le \frac{|l(p)|^2}{\pi^2} \int_{l(p)} \left( \frac{\partial \vp_W}{\partial x} \right)^2
\le \frac{4}{\pi^2} \int_{l(p)} \left( \frac{\partial \vp_W}{\partial x} \right)^2
$$
Hence
\begin{align*}
\int_{W} \vp_W^2
& = \int_{W_x} \left(\int_{l(p)} \vp_W^2 \right) \D p
\le \frac{4}{\pi^2} \int_{W_x} \left( \int_{l(p)} \left( \frac{\partial \vp_W}{\partial x} \right)^2\right) \D p \\
& = \frac{4}{\pi^2} \int_{W} \left( \frac{\partial \vp_W}{\partial x} \right)^2
\le \frac{4}{\pi^2} \int_{W} |\nabla \vp_W|^2 . \qedhere
\end{align*}
\end{proof}

The following lemma is analogous to Lemma 2.5 in \cite{JN2000}.

\begin{lemma}
\label{disteigenvalues}
There is a constant $C(\varepsilon,M,H)$ such that if $W_1, W_2 \in \calW(\varepsilon,M,H,1)$ and $\nu_{W_1}$, $\nu_{W_2}$ are corresponding eigenvalues $\nu$ for domains $W_1$, $W_2$ then 
$$
|\nu_{W_1} - \nu_{W_2}| \le C(\varepsilon,M,H) d(W_1,W_2)^{1/3}.
$$
\end{lemma}
\begin{proof}
By Corollary \ref{inclusion}, $W_1, W_2 \in \calA(p_0,r_1,r_2,\eta)$ for $p_0 = (0,- \eps/(2M), 0)\in \R^3$ and some $0 < r_1 < r_2$ and $\eta$ depending only on $\varepsilon$, $M$, $H$. Using this, Lemma \ref{uniform} and Lemma \ref{H32} we obtain that $\varphi_i\in \Hs(W_i)$ and $\|\vp_{W_i}\|_{\Hs(W_i)} \le C_1 = C_1(\varepsilon,M,H)$ for $i = 1, 2$. Also by Lemma~\ref{H32} there exist extensions $\tilde{\vp}_{W_i}$ ($i = 1, 2$)  of functions $\vp_{W_i}$ such that $\supp(\tilde{\vp}_{W_i}) \subset B(p_0,r_2 + 1)$ and $\|\tilde{\vp}_{W_i}\|_{\Hs(\R^3)} \le C_2 = C_2(\varepsilon,M,H)$. By a symmetrization argument, we may also assume that these extensions $\tilde{\vp}_{W_i}$ are again odd functions of $x$. Lemma \ref{H32} gives also that $\|\nabla\tilde{\vp}_{W_i}\|_{L^3(\R^3)} \le C_3 = C_3(\varepsilon,M,H)$.

We have 
\begin{equation}
\label{W1nabla}
\int_{W_1} |\nabla \tilde{\vp}_{W_2}|^2 =
\int_{W_2} |\nabla {\vp}_{W_2}|^2 +
\int_{W_1 \setminus W_2} |\nabla \tilde{\vp}_{W_2}|^2 -
\int_{W_2 \setminus W_1} |\nabla {\vp}_{W_2}|^2.
\end{equation}
By H{\"o}lder's inequality (for $p = 3$, $q = 3/2$) we get
\begin{align*}
\int_{W_1 \setminus W_2} & |\nabla \tilde{\vp}_{W_2}|^2
= \int_{\R^3} 1_{W_1 \setminus W_2} |\nabla \tilde{\vp}_{W_2}|^2 
\le \|1_{W_1 \setminus W_2}\|_{L^3(\R^3)} \| \, |\nabla \tilde{\vp}_{W_2}|^2 \|_{L^{3/2}(\R^3)} \\
& = |W_1 \setminus W_2|^{1/3} \|\nabla \tilde{\vp}_{W_2}\|_{L^3(\R^3)}^2
\le C_4 |W_1 \setminus W_2|^{1/3} 
\le C_5 d(W_1,W_2)^{1/3},
\end{align*}
where $C_4$, $C_5$ are constants depending only on $\varepsilon$, $M$ and $H$.

Note also that $\int_{W_2} |\nabla {\vp}_{W_2}|^2 = \nu_{W_2}$ and 
$
\int_F \tilde{\vp}^2_{W_2} = 
\int_F {\vp}^2_{W_2} = 1
$.
By the variational characterization of $\nu$ and (\ref{W1nabla}) it follows that 
$$
\nu_{W_1} \le \frac{\int_{W_1} |\nabla \tilde{\vp}_{W_2}|^2}{\int_F \tilde{\vp}^2_{W_2}(x,0,z) \, dx \, dz} 
\le \nu_{W_2} + C_5 d(W_1,W_2)^{1/3}.
$$
In the similar way we get $\nu_{W_2} \le \nu_{W_1} + C_5 d(W_1,W_2)^{1/3}$.
\end{proof}

Let $W \in \calW(\varepsilon,M,H,1)$. Recall that the eigenfunction $\vp_W$ has the lowest eigenvalue $\nu_W$ among all the eigenfunctions of the problem (\ref{lap} - \ref{ort}) on $W$ which are odd functions of $x$, and the next such eigenvalue is strictly greater. We  denote it by $\nu_W + \ve_W$, where $\ve_W > 0$ is a constant depending on $W$. Hence if $g \in H^1(W)$, $\int_{F} \vp_W g = 0$ and $g$ is odd with respect to $x$, then
\begin{equation}
\label{perpg}
\int_W |\nabla g|^2 \ge (\nu_W + \ve_W) \int_F g^2 .
\end{equation}

\begin{lemma}
\label{W1W2}
Let $W_1, W_2 \in \calW(\varepsilon,M,H,1)$. There is a constant $C= C(\ve,\ve_{W_1},M,H)$ such that
\begin{align}
\label{phi1phi2A}
& \int_F |\vp_{W_1} - \vp_{W_2}|^2
\le C d(W_1,W_2)^{1/3},
\\ 
\label{phi1phi2B}
& 1 \ge \int_F \vp_{W_1} \vp_{W_2} \ge
1 - \frac{C}{2} d(W_1,W_2)^{1/3}.
\end{align}
\end{lemma}

\begin{proof}
Recall that $\vp_{W_i}$ ($i = 1, 2$) are normalized so that $\int_F |\vp_{W_i}|^2 = 1$, $\vp_{W_i} > 0$ on $W_+$ and $\vp_{W_i} < 0$ on $W_-$, where $W_+ = \{(x,y,z) \in W: \, x > 0\}$, $W_- = \{(x,y,z) \in W: \, x < 0\}$. 

First note that (\ref{phi1phi2B}) follows easily from (\ref{phi1phi2A}) and the equality
$$
2 \int_F \vp_{W_1} \vp_{W_2} = 
- \int_F (\vp_{W_1} - \vp_{W_2})^2 + \int_F \vp_{W_1}^2 + \int_F \vp_{W_2}^2 
= 2 - \int_F (\vp_{W_1} - \vp_{W_2})^2.
$$
Hence, it is sufficient to show (\ref{phi1phi2A}). Put $\alpha_1 = \int_F \vp_{W_1} \vp_{W_2}$. By the normalization of $\vp_{W_1}$ and $\vp_{W_2}$ we have $0 < \alpha_1 \le 1$. Let $g = \tilde{\vp}_{W_2} - \alpha_1 \tilde{\vp}_{W_1}$, where $\tilde{\vp}_{W_1}$, $\tilde{\vp}_{W_2}$ are the extensions of ${\vp}_{W_1}$, ${\vp}_{W_2}$ defined in the proof of Lemma \ref{disteigenvalues}. On $W_1 \cap W_2$, $\tilde{\vp}_{W_i} = {\vp}_{W_i}$ are continuous ($i = 1, 2$), and so also $g$ is continuous on $W_1 \cap W_2$. In particular $g$ is continuous on $F$. We also have $g \in L^2(F)$, and by the definition of $\alpha_1$ we have 
\begin{equation}
\label{W1g}
\int_F \vp_{W_1} g = 0.
\end{equation}
Furthermore, $g \in \Hs(W_1) \subset H^1(W_1)$. By arguments as in the proof of Lemma \ref{disteigenvalues},
\begin{equation}
\label{nablaW1g}
\begin{aligned}
\nu_{W_2} & = \int_{W_2} |\nabla {\vp}_{W_2}|^2
= \int_{W_1} |\nabla \tilde{\vp}_{W_2}|^2 +
\int_{W_2 \setminus W_1} |\nabla \tilde{\vp}_{W_2}|^2 -
\int_{W_1 \setminus W_2} |\nabla \tilde{\vp}_{W_2}|^2 \\
& \ge \int_{W_1} |\nabla \tilde{\vp}_{W_2}|^2 - C_1 d(W_1,W_2)^{1/3} \\
& = \int_{W_1} |\nabla(\alpha_1 \vp_{W_1} + g)|^2 - C_1 d(W_1,W_2)^{1/3} \\
& = \alpha_1^2 \int_{W_1} |\nabla \tilde{\vp}_{W_1}|^2 
+ 2 \alpha_1 \int_{W_1} \nabla \tilde{\vp}_{W_1} \nabla g
+ \int_{W_1} |\nabla g|^2 - C_1 d(W_1,W_2)^{1/3},
\end{aligned}
\end{equation}
where $C_1 = C_1(\ve,M,H)$. We have $\int_{W_1} |\nabla \tilde{\vp}_{W_1}|^2 = \nu_{W_1}$, and by the Green's formula and (\ref{W1g}) we get
\begin{align*}
\int_{W_1} \nabla \tilde{\vp}_{W_1} \nabla g 
&= \int_{B_1} \left(\frac{\partial \vp_{W_1}}{\partial n}\right) g
+ \int_{F} \left(\frac{\partial \vp_{W_1}}{\partial y}\right) g 
- \int_{W_1} (\Delta \vp_{W_1}) g \\
&= \nu_{W_1} \int_F \vp_{W_1} g = 0.
\end{align*}
Since $g$ is odd with respect to $x$ (because $\tilde{\vp}_{W_1}$ and $\tilde{\vp}_{W_2}$ are odd) and  $g \in H^1(W_1)$, by (\ref{W1g}) and (\ref{perpg}) we get
$$
\int_{W_1} |\nabla g|^2 \ge (\nu_{W_1} + \ve_{W_1}) \int_F g^2 .
$$
Put $\alpha_2 = (\int_F g^2)^{1/2}$. By the definition of $\alpha_1$ and the normalization of $\vp_{W_2}$ we have $\alpha_1^2 + \alpha_2^2 = 1$. From (\ref{nablaW1g}) we get
\begin{align*}
\nu_{W_2} &\ge 
\alpha_1^2 \nu_{W_1} + 
\alpha_2^2 (\nu_{W_1} + \ve_{W_1})
 - C_1 d(W_1,W_2)^{1/3} \\
&= \nu_{W_1} + \alpha_2^2 \ve_{W_1} - C_1 d(W_1,W_2)^{1/3},
\end{align*}
where $C_1 = C_1(\ve,M,H)$. Using this and Lemma \ref{disteigenvalues} we obtain
$$
\alpha_2^2 \le \frac{\nu_{W_2} - \nu_{W_1} + C_1 d(W_1,W_2)^{1/3}}{\ve_{W_1}} 
\le \frac{C_2 d(W_1,W_2)^{1/3}}{\ve_{W_1}},
$$
where $C_2 = C_2(\ve,M,H)$. Finally we have
\[
 \int_F |\vp_{W_1} - \vp_{W_2}|^2 = 2 - 2 \alpha_1 \le 2 - 2 \alpha_1^2
= 2 \alpha_2^2 \le \frac{2 C_2 d(W_1,W_2)^{1/3}}{\ve_{W_1}}.
\]
Here we used the fact that $0 \le \alpha_1 \le 1$.
\end{proof}

\begin{lemma}
\label{lem:nablaW1W2}
Let $W_1, W_2 \in \calW(\ve,M,H,1)$. There is a constant $C = C(\ve,\ve_{W_1},M,H)$ such that
\begin{equation}
\label{eq:nablaW1W2}
\int_{W_1 \cap W_2} |\nabla(\vp_{W_1} - \vp_{W_2})|^2 \le 
C d(W_1,W_2)^{1/3}.
\end{equation}
\end{lemma}

\begin{proof}
We have
$$
\int_{W_1 \cap W_2} |\nabla(\vp_{W_1} - \vp_{W_2})|^2
\le \int_{W_1} |\nabla \vp_{W_1}|^2 
+ \int_{W_1} |\nabla \tilde{\vp}_{W_2}|^2
- 2 \int_{W_1} \nabla \vp_{W_1} \nabla \tilde{\vp}_{W_2}.
$$
Note that $\int_{W_1} |\nabla \vp_{W_1}|^2 = \nu_{W_1}$ and, as in the proof of Lemma \ref{disteigenvalues},
\begin{align*}
\int_{W_1} |\nabla \tilde{\vp}_{W_2}|^2
&=
\int_{W_2} |\nabla \tilde{\vp}_{W_2}|^2 -
\int_{W_2 \setminus W_1} |\nabla \tilde{\vp}_{W_2}|^2 +
\int_{W_1 \setminus W_2} |\nabla \tilde{\vp}_{W_2}|^2 \\
&\le \nu_{W_2} + C_1 d(W_1,W_2)^{1/3} \\
&\le \nu_{W_1} + C_2 d(W_1,W_2)^{1/3}
\end{align*}
where $C_1 = C_1(\ve,M,H)$ and $C_2 = C_2(\ve,M,H)$. By the Green's formula and Lemma ~\ref{W1W2} we obtain
\begin{align*}
\int_{W_1} \nabla \vp_{W_1} \nabla \tilde{\vp}_{W_2}
&= \int_{B_1} \left(\frac{\partial \vp_{W_1}}{\partial n}\right) \tilde{\vp}_{W_2}
+ \int_{F} \left(\frac{\partial \vp_{W_1}}{\partial y}\right) \tilde{\vp}_{W_2} 
- \int_{W_1} (\Delta \vp_{W_1}) \tilde{\vp}_{W_2} \\
&= \nu_{W_1} \int_F \vp_{W_1} \vp_{W_2} \\
&\ge \nu_{W_1} - \frac{\nu_{W_1} C_3}{2} d(W_1,W_2)^{1/3} \\
&\ge \nu_{W_1} - C_4 d(W_1,W_2)^{1/3},
\end{align*}
where $C_3 = C_3(\ve,\ve_{W_1},M,H)$ and $C_4 = C_4(\ve,\ve_{W_1},M,H)$. The last inequality follows from Lemma \ref{uniform}. Finally, we obtain 
$$
\int_{W_1 \cap W_2} |\nabla(\vp_{W_1} - \vp_{W_2})|^2
\le 2 \nu_{W_1} + C_2 d(W_1,W_2)^{1/3} - 2 \nu_{W_1} + 2 C_4 d(W_1,W_2)^{1/3}
$$
which gives the assertion of the lemma.
\end{proof}

\begin{lemma}
\label{lemW1W2}
Let $W_1, W_2 \in \calW(\ve,M,H,1)$. There is a constant $C = C(\ve,\ve_{W_1},M,H)$ such that
\begin{equation}
\label{nablaW1W2}
\int_{W_1 \cap W_2} |\vp_{W_1} - \vp_{W_2}|^2 \le 
C d(W_1,W_2)^{1/3}.
\end{equation}
\end{lemma}
\begin{proof}
Note that $W_1 \cap W_2$ is symmetric with respect to the $yz$-plane, and its intersections with lines parallel to the $x$-axis are intervals. Since $\vp_{W_1}$, $\vp_{W_2}$ are odd with respect to the $x$-variable, as in the proof of Lemma \ref{uniform} we get
$$
\int_{W_1 \cap W_2} |\vp_{W_1} - \vp_{W_2}|^2 
\le \frac{4}{\pi^2} \int_{W_1 \cap W_2} |\nabla(\vp_{W_1} - \vp_{W_2})|^2.
$$
Now the assertion of the lemma follows from Lemma \ref{lem:nablaW1W2}.
\end{proof}

The next lemma is analogous to Lemma 2.6 in \cite{JN2000}.

\begin{lemma}
\label{familyWt}
Let $W_t \in \calW(\ve,M,H,1)$ for $t$ in a neighborhood of $0$, and suppose that $d(W_t, W_0) \to 0$ as $t \to 0$. Let $K$ be a compact subset of $W_0$. Then there exists $\delta > 0$ such that for all $|t| < \delta$ we have
\begin{align*}
 & K \subset W_t && \text{and} && \|(\vp_{W_t} - \vp_{W_0}) 1_K\|_{\infty} \to 0 \quad \text{as} \quad t \to 0.
\end{align*}
\end{lemma}

\begin{proof}
Since $K$ is compact we have $\dist(K,\partial{W_0}) = 2 r$ for some $r > 0$. Since $d(W_t,W_0) \to 0$ there exists $\delta > 0$ such that for all $|t| < \delta$ we have $K \subset W_t$ and $\dist(K,\partial{W_t}) > r$. 

Let $|t| < \delta$ and $p \in K$. Then $B(p,r) \subset W_t \cap W_0$ and $\vp_{W_t}$, $\vp_{W_0}$ are harmonic in $B(p,r)$. Using this and Lemma \ref{lemW1W2} we get
\begin{align*}
|\vp_{W_t}(p) - \vp_{W_0}(p)| 
&\le \frac{1}{|B(p,r)|} \int_{B(p,r)} |\vp_{W_t} - \vp_{W_0}| \\
&\le \frac{1}{\sqrt{|B(p,r)|}} \left( \int_{B(p,r)} |\vp_{W_t} - \vp_{W_0}|^2 \right)^{1/2} \\
&\le \frac{1}{\sqrt{|B(p,r)|}} \left( \int_{W_t \cap W_0} |\vp_{W_t} - \vp_{W_0}|^2 \right)^{1/2} \\
&\le C(\ve,\ve_{W_0},M,H) r^{-3/2} d(W_t,W_0)^{1/6},
\end{align*}
which implies the assertion of the lemma.
\end{proof}

\section{Monotonicity of the odd eigenfunction for some class of piecewise smooth domains}

\begin{figure}
\centering
\includegraphics[scale=0.8]{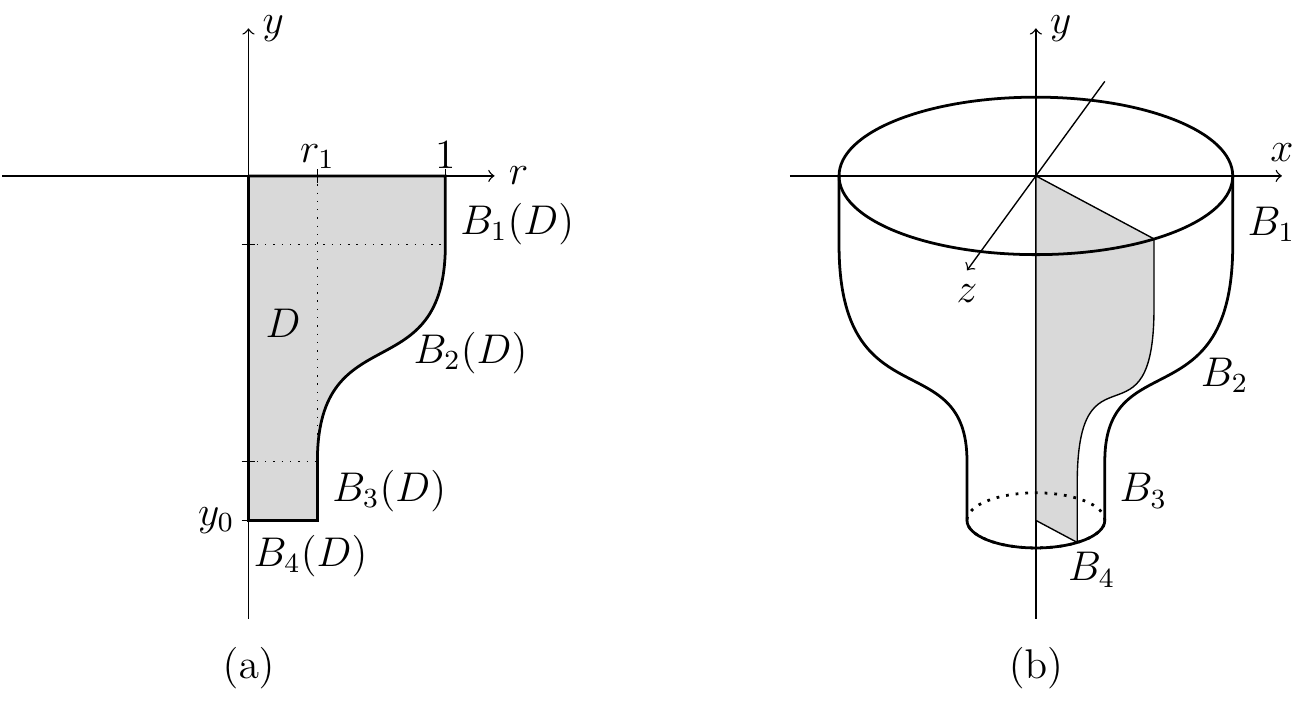}
\caption{(a) An example of a domain $D$ in the class $\calD_1$. (b) An example of a domain $W$ in the class $\calW_1$.}
\label{fig:4}
\end{figure}

This section is similar to Section 3 in \cite{JN2000}. The aim of this section is to show Theorem \ref{th:oddmon} for domains in some special subclass $\calW_1$ of the class $\calW$, defined below (see Figure~\ref{fig:4}(b)). First we need to define the subclass $\calD_1$ of the class $\calD$. In the whole section we use notation~\eqref{psi11}.

\begin{definition}
\label{classD1}
Let $y_0 < 0$, $r_1 \in (0,1)$, $T = r_1 - y_0$, $r_1 = T_1 < T_2 < T_3 < T$. The domain $D \subset \{(r,y): \, r > 0,\, y < 0\}$ belongs to the class of domains $\calD_1$ iff its boundary consists of the following 3 parts (see Figure~\ref{fig:4}(a)):
\begin{enumerate}
\item[(i)] the horizontal interval $F(D) = \{(r,y): \, r \in [0,1), \, y = 0\}$;
\item[(ii)] the vertical interval $R(D) = \{(r,y): \, r = 0, \, y \in (y_0,0]\}$;
\item[(iii)]  $B(D)$, parametrized by a simple continuous curve $(r(t), y(t))$, $t \in [0,T]$, satisfying the following conditions:
\begin{enumerate}
\item $(r(t), y(t)) = (t,y_0)$ for $t \in [0,T_1]$,
\item $(r(t), y(t)) = (r_1,y_0 + t - T_1)$ for $t \in [T_1,T_2]$,
\item $y(t) = y_0 + t - T_1$ for $t \in [T_2,T_3]$, $r(t)$ is a strictly increasing function on $[T_2,T_3]$,
\item $(r(t),y(t)) = (1,y_0 + t - T_1)$ for $t \in [T_3,T]$,
\item $r(t)$ is a $C^{\infty}$ function on $[T_1,T]$.
\end{enumerate}
\end{enumerate}
We denote by $B_1(D)$, $B_2(D)$, $B_3(D)$, $B_4(D)$, the parts of $B(D)$  corresponding to $t \in [T_3,T)$, $t \in (T_2,T_3)$, $t \in (T_1,T_2]$, $t \in (0,T_1]$  respectively.
\end{definition}

\begin{definition}
\label{def:classW1}
The domain $W$ belongs to the class of domains $\calW_1$ iff $W = W(D)$ for some $D \in \calD_1$. For $W = W(D) \in \calW_1$ we denote $B_i = \{(x,y,z) \in \R^3: \, (\sqrt{x^2 + z^2},y) \in B_i(D)\}$, $i = 1, 2, 3, 4$.
\end{definition}

Note that for $\eps = \min( r_1,T - T_3)$, $H = - y_0 + 1$, the domain $D \in \calD_1$ belongs to the class $\calD(\eps,1,H,1)$  and $W(D) \in \calW(\ve,1,H,1)$. 

First we need the following general lemma.
\begin{lemma}
\label{lem:newcoordinates}
Let $R > 0$, $( \hat{x},\hat{y})$ be a rectangular coordinate system in $\R^2$ and $B(0,R) = \{( \hat{x},\hat{y}): \,  \hat{x}^2 + \hat{y}^2 < R^2\}$. Let $f \in C^{2,1}(-R,R)$ be such that $f(0) = f'(0) = 0$ and $\Gamma(f) = \{(x,f(x)): x \in (-R,R)\} \cap B(0,R)$ be the  part of the graph of $f$ contained in $B(0,R)$. Assume that $g \in C^{2,1}(B(0,R))$ and $\nabla g(0,0) = 0$. 

If $\frac{\partial g}{\partial n}( \hat{x}, \hat{y}) = 0$ for all $( \hat{x}, \hat{y}) \in \Gamma(f)$, where $\frac{\partial}{\partial n}$ is the normal derivative to $\Gamma(f)$ at point $( \hat{x}, \hat{y}) \in \Gamma(f)$, then
\begin{equation}
\label{eq:mixder}
\frac{\partial^2 g}{\partial  \hat{x} \partial  \hat{y}}(0,0) = 0.
\end{equation}
\end{lemma}
\begin{proof}
By assumptions on $f$ there exists $r > 0$ such that if $ \hat{x} \in (-r,r)$ then $( \hat{x},f( \hat{x})) \in B(0,R)$ so $( \hat{x},f( \hat{x})) \in \Gamma(f)$. In the whole proof we will assume that $ \hat{x} \in (-r,r)$. 

The unit (upper) normal derivative to $\Gamma(f)$ at $( \hat{x},f( \hat{x})) \in \Gamma(f)$ is equal to 
$$
\frac{\partial}{\partial n} = \frac{1}{\sqrt{1 + (f'( \hat{x}))^2}} \left(- f'( \hat{x}) \frac{\partial}{\partial  \hat{x}} + \frac{\partial}{\partial  \hat{y}}\right).
$$
The condition $\frac{\partial g}{\partial n}( \hat{x}, \hat{y}) = 0$ gives 
\begin{equation}
\label{eq:yder}
\frac{\partial g}{\partial  \hat{y}}( \hat{x},f( \hat{x})) = f'( \hat{x}) \frac{\partial g}{\partial  \hat{x}}( \hat{x},f( \hat{x})).
\end{equation}
By our assumptions on $f$ we get $f( \hat{x}) = O( \hat{x}^2)$ and $f'( \hat{x}) = O(| \hat{x}|)$. Similarly, assumptions on $g$ give $\frac{\partial g}{\partial  \hat{x}}( \hat{x}, \hat{y}) = O(\sqrt{ \hat{x}^2 +  \hat{y}^2})$ so $\frac{\partial g}{\partial  \hat{x}}( \hat{x},f( \hat{x})) = O(| \hat{x}|)$. By (\ref{eq:yder}) we get
\begin{equation}
\label{eq:yder0}
\frac{\partial g}{\partial  \hat{y}}( \hat{x},f( \hat{x})) = O( \hat{x}^2).
\end{equation}
By Taylor expansion for $\frac{\partial g}{\partial  \hat{y}}$ we get for $( \hat{x}, \hat{y}) \in B(0,R)$
$$
\frac{\partial g}{\partial  \hat{y}}( \hat{x}, \hat{y}) 
= \frac{\partial g}{\partial  \hat{y}}(0,0) 
+ \frac{\partial^2 g}{\partial  \hat{x} \partial  \hat{y}}(0,0)  \hat{x}
+ \frac{\partial^2 g}{\partial  \hat{y}^2}(0,0)  \hat{y}
+ o(\sqrt{ \hat{x}^2 +  \hat{y}^2}).
$$
It follows that 
$$
\frac{\partial g}{\partial  \hat{y}}( \hat{x},f( \hat{x})) 
= \frac{\partial^2 g}{\partial  \hat{x} \partial  \hat{y}}(0,0)  \hat{x}
+ \frac{\partial^2 g}{\partial  \hat{y}^2}(0,0) f( \hat{x})
+ o(| \hat{x}|).
$$
By (\ref{eq:yder0}) and $f( \hat{x}) = O( \hat{x}^2)$ we get
$$
\frac{\partial^2 g}{\partial  \hat{x} \partial  \hat{y}}(0,0)  \hat{x} = o(| \hat{x}|). 
$$
This implies (\ref{eq:mixder}).
\end{proof}

\begin{lemma}
\label{lem:smooth}
Let $D \in \calD_1$ and $W = W(D) \in \calW_1$. We have $\vp \in C^{2,1}(\overline{W})$.
\end{lemma}
\begin{proof}
This is a standard result. The lemma follows from \cite[formula~(13.14)]{LBK1984} and \cite[p.~63, lines~1--3]{LBK1984}, see also~\cite{S2008}.
\end{proof}

As an immediate conlusion of this lemma we obtain:

\begin{corollary}
Let $D \in \calD_1$ and $W = W(D) \in \calW_1$. Then $\frac{\partial \vp}{\partial x}$, $\frac{\partial \vp}{\partial y}$ are bounded on $\overline{W}$. 
\end{corollary}

Recall our notation $W_+ = \{(x,y,z) \in W: \, x > 0\}$. Recall also that we may assume that $\psi > 0$ on $D$.

\begin{lemma}
\label{lem:global}
Let $W  = W(D)$ be an axisymmetric liquid domain. If $\frac{\partial \psi}{\partial r} \ge 0$ on $D$ and $\frac{\partial \psi}{\partial y} \ge 0$ on $D$ then $\frac{\partial \vp}{\partial y} \ge 0$ on $W_+$ and $\frac{\partial \vp}{\partial x} \ge 0$ on $W$.
\end{lemma}

\begin{proof}
In the cylindrical coordinates $(r,\theta,y)$ (see \eqref{cylindrical}) we have $W_+ = \{ \theta \in (-\pi/2, \pi/2) \}$. Clearly, 
$$
\frac{\partial \vp}{\partial y} =
\frac{\partial \psi}{\partial y} (r,y) \cos \theta.
$$
Hence $\frac{\partial \psi}{\partial y} \ge 0$ on $D$ implies $\frac{\partial \vp}{\partial y} \ge 0$ on $W_+$.  Furthermore,
$$
\frac{\partial \vp}{\partial x} 
= \frac{\partial \vp}{\partial r} \cos \theta
+ \frac{\partial \vp}{\partial \theta} \left(\frac{-\sin \theta}{r} \right).
$$
Since $\vp = \psi(r,y) \cos \theta$, we obtain 
\begin{equation}
\label{eq:psixphix}
\frac{\partial \vp}{\partial x}
= \frac{\partial \psi}{\partial r}(r,y) \cos^2 \theta
+ \psi(r,y) \left(\frac{\sin^2 \theta}{r} \right).
\end{equation}
It follows that  $\psi \ge 0$ and $\frac{\partial \psi}{\partial r} \ge 0$ on $D$ implies $\frac{\partial \vp}{\partial x} \ge 0$ on $W$. 
\end{proof} 

Note that if $D \in \calD_1$ and $W = W(D) \in \calW_1$ then, in view of Lemma \ref{lem:smooth}, (\ref{eq:psixphix}) holds  on $\overline{W}$, given that $r = \sqrt{x^2 + z^2} \ne 0$.

\begin{lemma}
\label{lem:phipsi}
Let $D \in \calD_1$ and $W = W(D) \in \calW_1$. Assume that $(r,y) \in D$. For  $x = r$ and any $n,m \in \N$ we have
$$
\frac{\partial^{n + m} \vp}{\partial x^n \partial y^m}(x,y,0) 
= \frac{\partial^{n + m} \psi}{\partial r^n \partial y^m}(r,y). 
$$
\end{lemma}
\begin{proof}
This follows from the fact that for $x = r$ we have $\vp(x,y,0) = \psi(r,y)$.
\end{proof}

The proof of the following lemma uses some ideas from the proof of Proposition 3.2 in \cite{JN2000}.
\begin{lemma}
\label{lem:yderivative}
Let $D \in \calD_1$ and $W = W(D) \in \calW_1$. If $\frac{\partial \psi}{\partial r} \ge 0$ on $D$ and $\frac{\partial \psi}{\partial y} \ge 0$ on $D$ then $\frac{\partial \psi}{\partial y} > 0$ on $B_1(D) \cup B_2(D) \cup B_3(D)$ (see Figure~\ref{fig:4}(a)).
\end{lemma}

\begin{proof}
Since $\frac{\partial \psi}{\partial y} \ge 0$ on $D$ we have $\frac{\partial \psi}{\partial y} \ge 0$ on $B_1(D) \cup B_2(D) \cup B_3(D)$ so we only need to exclude the possibility that $\frac{\partial \psi}{\partial y} = 0$ at some point of $B_1(D) \cup B_2(D) \cup B_3(D)$. 

\begin{figure}
\centering
\includegraphics[scale=0.8]{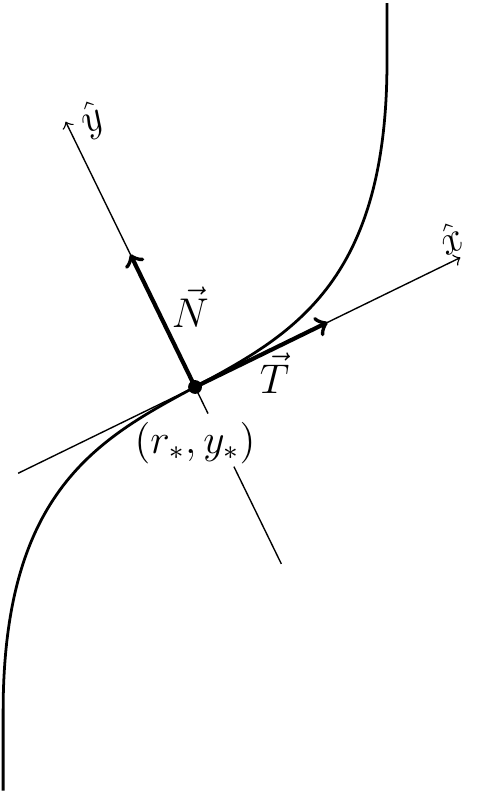}
\caption{New coordinate system $( \hat{x}, \hat{y})$.}
\label{fig:5}
\end{figure}

On the contrary assume that $\frac{\partial \psi}{\partial y}(r_*,y_*) = 0$ at some point $(r_*,y_*) \in B_1(D) \cup B_2(D) \cup B_3(D)$.  Since $\frac{\partial \psi}{\partial  n}(r_*,y_*) = 0$, $\frac{\partial \psi}{\partial y}(r_*,y_*) = 0$ and since  the normal vector is not parallel to  the $y$-axis, we have $\nabla \psi(r_*,y_*) = 0$. Let $\vec{T} = (\alpha,\beta) \in \R^2$, $\alpha \ge 0$, $\beta > 0$, $\alpha^2 + \beta^2 = 1$ be the tangent unit vector to $B(D)$ at point $(r_*,y_*)$ and $\vec{N} = (-\beta,\alpha) \in \R^2$ be the normal unit vector to  $B(D)$ at point $(r_*,y_*)$ (see Figure~5). 
Let us introduce a rectangular coordinate system $(\hat{x},\hat{y})$ with origin at $(r_*,y_*)$ so that $\hat{x}$-axis is in the direction $\vec{T}$  and $\hat{y}$-axis is in the direction $\vec{N}$.
Note that in these new coordinates $B(D)$ in some ball centered in the origin is the graph of  a $C^{2,1}$ (even $C^\infty$) function $f$ such that $f(0) = f'(0) = 0$.

By Lemma \ref{lem:smooth} $\psi \in C^{2,1}(\overline{D})$. Using this by \cite[Lemma 6.37]{GT1977} there is a  $C^{2,1}$ extension of the function $ \psi$ to some ball centered at $(r_*,y_*)$. We will denote this extension by the same letter $\psi$.
 
Then the assumptions of Lemma \ref{lem:newcoordinates} are satisfied which gives
\begin{equation}
\label{eq:mixder1}
 \frac{\partial^2  \psi}{\partial  \hat{x} \partial  \hat{y}}( r_*, y_*) = 0.
\end{equation}
Since $\nabla \psi(r_*,y_*) = 0$ we get $\frac{\partial \psi}{\partial  \hat{x}}(r_*,y_*) = 0$.

Note that $\pd{\hat{x}} = \alpha \pd{r} + \beta \pd{y}$ and define $u = \alpha \pd[\vp]{x} + \beta \pd[\vp]{y}$. By Lemma \ref{lem:phipsi}, $u(x, y, 0) = \pd[\psi]{\hat{x}}(r, y)$ (with $r = x$). Note that $u$ is harmonic in $W$ and by Lemma \ref{lem:global} and the fact that $\alpha \ge 0$, $\beta > 0$ we have $u \ge 0$ in $W_+$.  Finally, for $x_* = r_*$,
$$
 u(x_*, y_*, 0)
= \frac{\partial \psi}{\partial  \hat{x}}(r_*,y_*) = 0
$$
and
\begin{equation}
\label{eq:normal}
 \pd[u]{n}(x_*, y_*, 0)
= \frac{\partial^2 \psi}{\partial  \hat{x} \partial  \hat{y} }(r_*,y_*) = 0.
\end{equation}
At the point $(x_*,y_*,0) \in \partial W_+$ the domain $W_+$ satisfies the inner ball condition. Hence (\ref{eq:normal}) gives contradiction with the Hopf's lemma for the harmonic function $u$.
\end{proof}

\begin{lemma}
\label{lem:normal}
Let $u$ be a function which is harmonic in a bounded domain $\Omega \subset\textbf{} \R^n$ and continuous in $\overline{\Omega}$. Let $Q_1 \subset \partial \Omega$ be such that for each $p \in Q_1$ the boundary $\partial \Omega $ has a tangent plane at $p$,  the outer normal derivative $\frac{\partial u}{\partial n}(p)$ exists and $\frac{\partial u}{\partial n}(p) = 0$. Let $Q_2 = \partial \Omega \setminus Q_1$. If $ u \ge 0$ on $Q_2$ then $u \ge 0$ on $\overline{\Omega}$.
\end{lemma}
\begin{proof}
On the contrary assume that there exists $p_0 \in \overline{\Omega}$ such that $u(p_0) < 0$. By the maximum principle for $u$ there exists $p_1 \in Q_1$ such that $0 > u(p_1) = \min_{p \in \overline{\Omega}} u(p)$. By the normal derivative lemma (see e.g. \cite{ES1992} Lemma 2.33) $\frac{\partial u}{\partial n}(p_1) < 0$ which gives a contradiction.
\end{proof}

Let us denote $\R^3_+ = \{(x,y,z): \, x > 0\}$, $B_+ = B \cap \R^3_+$, $F_+ = F \cap \R^3_+$, $B_{i+} = B_i \cap \R^3_+$, $i = 1, 2, 3, 4$.

\begin{lemma}
\label{lem:B2y}
Let $D \in \calD_1$ and $W = W(D) \in \calW_1$. If $\frac{\partial \psi}{\partial y} \ge 0$ on $B_2(D)$ then $\frac{\partial \psi}{\partial y} \ge 0$ on $D$ and $\frac{\partial \vp}{\partial y} \ge 0$ on $\overline{W_+}$.
\end{lemma}

\begin{proof}
We have
$$
\frac{\partial \vp}{\partial y}
= \frac{\partial \psi}{\partial y}(r,y) \cos \theta
$$
on $\overline{W}$, so by the assumption $\frac{\partial \psi}{\partial y} \ge 0$ on $B_2(D)$ we get $\frac{\partial \vp}{\partial y} \ge 0$ on $\overline{B_{2+}}$. 

Let us denote $u = \frac{\partial \vp}{\partial y}$.  Then $u$ is harmonic in $W_+$ and continuous in $\overline{W_+}$. We have $u \ge 0$ on $\overline{B_{2+}}$. We also have $u = 0$ on $\overline{B_{4+}}$. Since $\vp \equiv 0$ on $\{(x,y,z) \in \overline{W}: \, x = 0\}$ we get $u \equiv 0$ on $\{(x,y,z) \in \overline{W}: \, x = 0\}$. Recall that $\psi > 0$ on $D$ so $\vp \ge 0$ on $\overline{W_+}$. We get $u = \frac{\partial \vp}{\partial y} = \nu \vp \ge 0$ on $F_+$, so $u \ge 0$ on $\overline{F_{+}}$. Note also that $\frac{\partial u}{\partial n} = 0$ on $B_{1+} \cup B_{3+}$. Now the assertion of the lemma follows from Lemma \ref{lem:normal} for $\Omega = W_+$ and $Q_1 = B_{1+} \cup B_{3+}$.
\end{proof}

\begin{lemma}
\label{lem:nablauWF}
Let $D \in \calD(\varepsilon,M,H,1)$ and $W = W(D) \in \calW(\varepsilon,M,H,1)$. Assume that $u \in H^1(W)$, $u$ is odd with respect to the $x$-axis, $\alpha = \left(\int_F u^2 \right)^{1/2} > 0$ and $|\int_{F} \vp_W u| \ne \alpha$. Then we have
$$
\int_W |\nabla u|^2 > \nu_W \int_F u^2.
$$
\end{lemma}

\begin{proof}
Put $\beta = \int_{F} \vp_W u$ and $g = u - \beta \vp_W$. We have $\int_{F} \vp_W g = 0$ and $\int_F g^2 = \alpha^2 - \beta^2 \ne 0$. We also have
$$
\int_W |\nabla u|^2 = \int_W |\nabla g|^2 + \int_W |\beta \nabla \vp_W|^2 
+ 2 \int_W \nabla g \nabla \vp_W.
$$
By Green's formula one easily get $\int_W \nabla g \nabla \vp_W = 0$ (cf. proof of Lemma \ref{W1W2}). We have $\int_W |\nabla \vp_W|^2 = \nu_W$. Using (\ref{perpg}) we obtain $\int_W |\nabla g|^2 > \nu_W (\alpha^2 - \beta^2)$. This implies the assertion of the lemma.
\end{proof}

\begin{lemma}
\label{lem:B2x}
Let $D \in \calD_1$ and $W = W(D) \in \calW_1$. If $\frac{\partial \psi}{\partial y} \ge 0$ on $B_2(D)$  then $\frac{\partial \vp}{\partial x} \ge 0$ on $\overline{W}$.
\end{lemma}

\begin{proof}
Since $B_2(D)$ is a graph of an increasing function $y = g(r)$,
$\frac{\partial \psi}{\partial n} = 0$ on $B_2(D)$ and $\frac{\partial \psi}{\partial y} \ge 0$ on $B_2(D)$, we get $\frac{\partial \psi}{\partial r} \ge 0$ on $B_2(D)$. We also have  $\frac{\partial \psi}{\partial r} = 0$ on $B_1(D) \cup B_3(D)$ so $\frac{\partial \psi}{\partial r} \ge 0$ on $B_1(D) \cup B_2(D) \cup B_3(D)$. By (\ref{eq:psixphix}) $\frac{\partial \vp}{\partial x} \ge 0$ on $\overline{B_1} \cup \overline{B_2} \cup \overline{B_3}$. Let $W_- = \{(x,y,z) \in W: \, x < 0 \}$ and $W_0 = \{(x,y,z) \in W: \, x = 0 \}$. Since $\vp < 0$ on $W_-$ and $\vp > 0$ on $W_+$ we get $\frac{\partial \vp}{\partial x} \ge 0$ on $\overline{W_0}$.

Assume, contrary to the hypothesis of the lemma, that there exists $ (x_*,y_*,z_*) \in W$ such that $\frac{\partial \vp}{\partial x}(x_*,y_*,z_*) < 0$. Since $\frac{\partial \vp}{\partial x}(x,y,z) = \frac{\partial \vp}{\partial x}(-x,y,z)$ we may  assume that $( x_*,y_*,z_*) \in W_+$. Let $V_+$ be the connected component of the set $\{(x,y,z) \in W: \,  \frac{\partial \vp}{\partial x}(x,y,z) < 0\}$ containing $( x_*,y_*,z_*)$. Since $\frac{\partial \vp}{\partial x} \ge 0$ on $\overline{B_1} \cup \overline{B_2} \cup \overline{B_3} \cup \overline{W_0}$ we have $V_+ \subset W_+$. 

\begin{figure}
\centering
\includegraphics[scale=0.8]{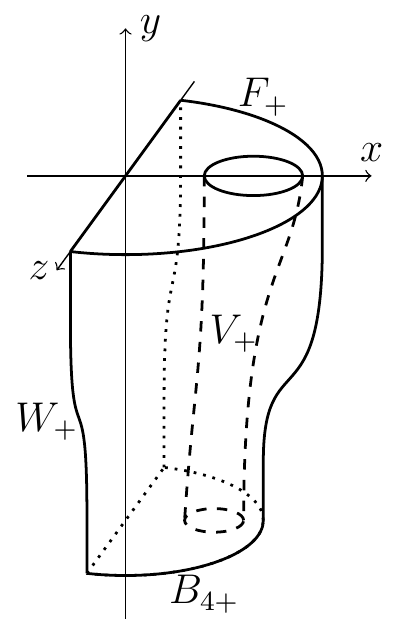}
\caption{The set $V_+$.}
\label{fig:6}
\end{figure}

Let 
\[
F(V_+) =  \relint (\partial{V_+} \cap F) 
\quad \text{and} \quad
B_4(V_+) =  \relintyzero (\partial{V_+} \cap B_4) 
\]
(see Figure~\ref{fig:6}). By  $\relint$, $\relintyzero$, we denote the  relative interior of a set in a $2$-dimensional plane. Recall that $-y_0$ is the depth of a liquid domain $D \in \calD_1$, see Definition~\ref{classD1}.

Since $\frac{\partial \vp}{\partial x} \ge 0$ on $\overline{B_1} \cup \overline{B_2} \cup \overline{B_3} \cup \overline{W_0}$ and $V_+$ is the connected component of the set $\{(x,y,z) \in W: \,  \frac{\partial \vp}{\partial x}(x,y,z) < 0\}$ we have $\frac{\partial \vp}{\partial x} = 0$ on
$\partial V_+ \setminus (F(V_+) \cup B_4(V_+))$. We also have
$$
\frac{\partial}{\partial n} \left(\frac{\partial \vp}{\partial x}\right) = 0
\quad \text{on} \quad B_4(V_+)
$$
and
$$
\frac{\partial}{\partial n} \left(\frac{\partial \vp}{\partial x}\right) =
\frac{\partial}{\partial y} \left(\frac{\partial \vp}{\partial x}\right) =
\nu_W \frac{\partial \vp}{\partial x}
\quad \text{on} \quad F(V_+).
$$
Let $V_- = \{(x,y,z) \in W: \,  (-x,y,z) \in V_+\}$. Of course $V_- \subset W_-$. Let us define the function $u$ in the following way. We put $u = \frac{\partial \vp}{\partial x}$ on $V_+$, $u = -\frac{\partial \vp}{\partial x}$ on $V_-$ and $u = 0$ on $\overline{W} \setminus (V_+ \cup V_-)$. We have 
$u \in C(\overline{W})$ and $u$ is odd with respect to the $x$-axis. By Green's formula we also get
\begin{equation}
\label{eq:nablau2}
\begin{aligned}
\int_W & |\nabla u|^2 
= 2 \int_{V_+} |\nabla u|^2 \\
&= 2 
\left( \int_{F(V_+)} \frac{\partial u}{\partial n} u
+ \int_{B_4(V_+)} \frac{\partial u}{\partial n} u
+ \int_{\partial V_+ \setminus (F(V_+) \cup B_4(V_+))} \frac{\partial u}{\partial n} u
- \int_{V_+} (\Delta u) u \right) \\
&= 2 \nu_W \int_{F(V_+)} u^2 = \nu_W \int_{F} u^2.
\end{aligned}
\end{equation}
We have $F(V_+) \ne \emptyset$ and $\int_F u^2 > 0$ because otherwise $u \equiv 0$ on $V_+$ which is impossible. Of course, there are points on $\relint(\partial{W_+} \cap F)$ for which $\frac{\partial \vp}{\partial x} > 0$, so $F(V_+) \ne \relint (\partial{W_+} \cap F)$. This implies that $u$ and $\vp_W$ are not linearly dependent  on $F$, so $|\int_F \varphi_W u| \ne |\int_F u^2|^{1/2}$.  By Lemma \ref{lem:nablauWF} we get 
$$
\int_W |\nabla u|^2 > \nu_W \int_{F} u^2,
$$
which contradicts \eqref{eq:nablau2}.
\end{proof}

\begin{lemma}
\label{lem:cylinder}
Let $D \in \calD_1$ and $W = W(D) \in \calW_1$. Assume that there exists $h > 0$ such that $W$ contains the cylinder $\{(x,y,z): \, x^2 + z^2 < 1, \, 0 > y > -h\}$. Then there exists $C(h)$ such that for all $(x,y,z) \in W$ such that $y \le - h$ we have
$$
|\vp(x,y,z)| \le C(h).
$$
\end{lemma}
\begin{proof}
Denote by $W'$ the set $\{(x, y, z) \in W : \, -h/2 > y \}$. Let $M$ be the supremum of $|\vp|$ over $W'$. By the maximum principle, the supremum is attained at the boundary of $W'$. By the normal derivative lemma (see \cite[Lemma 2.33]{ES1992}), since $\vp$ is not constant, it cannot attain its supremum or infimum at $\partial W' \cap \{(x,y,z): \, y < -h\}$, the part of the boundary where $\vp$ satisfies Neumann boundary condition. It follows that $M = |\vp(p_0)|$ for some $p_0 = (x_0, -h/2, z_0) \in \partial  W'$.

 By the Harnack inequality up to the Neumann boundary (see~\cite[Theorem~3.9]{BH1991}), there is $\delta = \delta(h) \in (0,h/2 \wedge 1)$ such that $\vp(p) \ge M/2$ for $p \in B(p_0, \delta) \cap W$. Furthermore, $|B(p_0, \delta) \cap W| \ge C_1 \delta^3$, where $C_1$ is an absolute constant. It follows that
\[
 \int_W \vp^2 \ge \int_{B(p_0, \delta) \cap W} \frac{M^2}{4} \ge \frac{C_1 \delta^3 M^2}{4} .
\]
On the other hand, by Lemma~\ref{uniform}, the left-hand side is bounded above by an absolute constant.
\end{proof}

\begin{figure}
\centering
\includegraphics[scale=0.8]{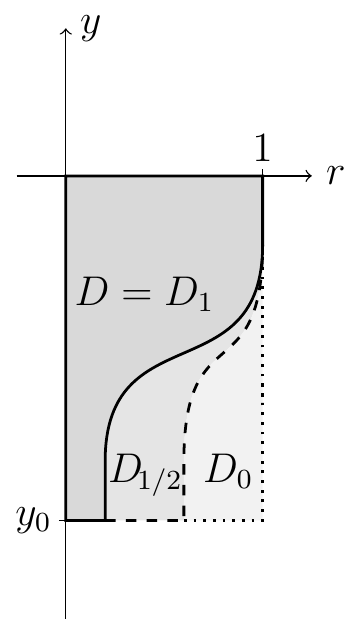}
\caption{Family of domains $D_s$.}
\label{fig:7}
\end{figure}

\begin{figure}
\centering
\includegraphics[scale=0.8]{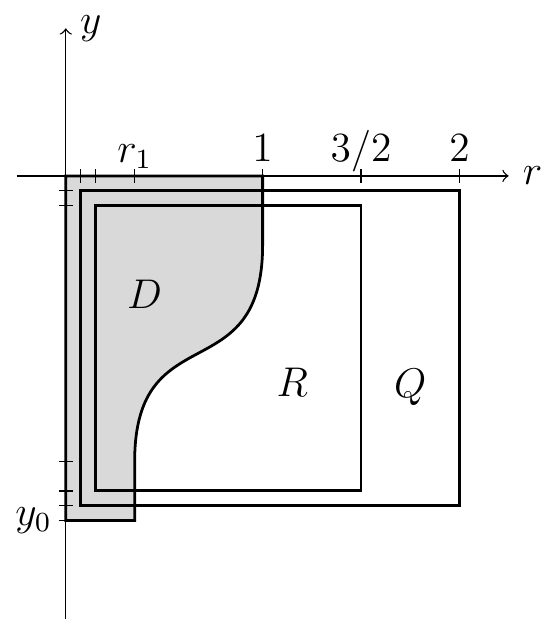}
\caption{An auxiliary picture for Lemma \ref{lem:uniformnorm}.}
\label{fig:8}
\end{figure}

For a fixed domain $D \in \calD_1$ we define the family of domains $D_s$, $s \in [0,1]$ such that $D_1 = D$ and $D_0$ is a rectangle $(0,1) \times (y_0,0)$ (see Figure~\ref{fig:7}).  Let $y_0$, $r_1$, $r(t)$, $y(t)$  be as in the definition of class $\calD_1$ for the domain $D$. For $y \in [y_0,0]$ put $g(y) = r(r_1 + y_0 - y)$. Note that $D = \{(r,y): \, y \in (y_0,0), \, r \in (0,g(y))\}$. The domain $D_s  \in \calD_1$ for $s \in (0,1)$ is defined by 
$$
D_s = \{(r,y): \, y \in (y_0,0), \, r \in (0,1 - s + sg(y))\}.
$$
 We denote $1 - s + s g(y) = g_s(y)$, so that $D_s = \{(r,y): y \in (y_0, 0) , \, r \in (0, g_s(y)) \}$.

\begin{lemma}
\label{lem:uniformnorm}
Let $D \in \calD_1$. We define a rectangle $R$ (see Figure~\ref{fig:8}) by
$$
R = (r_1/2,3/2) \times (y_1^R,y_2^R),
$$
where $y_1^R = y_0 + (T_2 - T_1)/2$, $y_2^R = -(T - T_3)/2$ and $r_1$, $y_0$, $T_1$, $T_2$, $T_3$, $T$ are taken from the definition of class $\calD_1$ for the domain $D$.
Let $D_s$, $s \in [0,1]$, be the family of domains constructed before the statement of this lemma for the domain $D$, and let $W_s = W(D_s)$. There exist a constant $a = a(D)$ such that for all $s \in [0,1]$
\begin{equation}
\label{eq:uniformnorm}
\sup_{p \in \overline{D_s \cap R}} 
\sum_{\alpha + \beta = 2} \left|\frac{\partial^{\alpha + \beta}}{\partial^{\alpha} r \partial^{\beta} y} \psi_{D_s}(p) \right|  \le a.
\end{equation}
Moreover, there exist constants $\ve = \ve(D)$, $H = H(D)$ such that for all $s, q \in [0,1]$ 
$$
W_s \in \calW(\ve,1,H,1) 
\quad \text{and} \quad
d(W_s,W_q) \to 0
\quad \text{as} \quad
s \to q.
$$
\end{lemma}
The key point in this lemma is that the constant $a$ does not depend on $s$. The proof of this result follows by elementary but tedious calculations.
\begin{proof}
We define a rectangle (see Figure~\ref{fig:8}) $Q = (r_1/4,2) \times (y_1^Q,y_2^Q)$, where $y_1^Q = y_0 + (T_2 - T_1)/4$ and $y_2^Q = -(T - T_3)/4$. For each $s \in [0,1]$ let us consider the following change of variables:
$$
\tilde{y} = y, \quad \tilde{r} = \frac{r}{g_s(y)}.
$$ 
This change of variables is chosen so that the curve $Q \cap \partial D_s$ is straightened in these new variables. In fact this transformation of variables changes $Q \cap \partial D_s$ to $\{(\tilde{r},\tilde{y}): \, \tilde{y} \in (y_1^Q,y_2^Q), \, \tilde{r} = 1\}$.

 Since $g_s(y) = g_s(\tilde{y})$, $g_s'(y) = s g'(\tilde{y})$, and $\pd{y} (\frac{r}{g_s(y)}) = - \frac{r s g'(y)}{(g_s(y))^2} = -\frac{\tilde{r} s g'(\tilde{y})}{g_s(\tilde{y})}$, we have 
\begin{align*}
\frac{\partial}{\partial r} & = \frac{1}{ g_s(\tilde{y})} \, \frac{\partial}{\partial \tilde{r}}, &
\frac{\partial}{\partial y} & = \frac{\partial}{\partial \tilde{y}} 
 - \frac{\tilde{r} s g'(\tilde{y})}{g_s(\tilde{y})} \, \frac{\partial}{\partial \tilde{r}}, &
\frac{\partial^2}{\partial r^2} & = \frac{1}{(g_s(\tilde{y}))^2} \, \frac{\partial^2}{\partial \tilde{r}^2},
\end{align*}
\begin{align*}
\frac{\partial^2}{\partial y^2} & = \frac{\partial^2}{\partial \tilde{y}^2}
- \frac{2 \tilde{r} s g'(\tilde{y})}{ g_s(\tilde{y})} \frac{\partial^2}{\partial \tilde{y} \partial \tilde{r}} +
\left(\frac{\tilde{r} s g'(\tilde{y})}{ g_s(\tilde{y})}\right)^2
\frac{\partial^2}{\partial \tilde{r}^2}
+
\frac{2 \tilde{r} s^2 (g'(\tilde{y}))^2 - \tilde{r} s g''(\tilde{y})  g_s(\tilde{y})}{( g_s(\tilde{y}))^2} \frac{\partial}{\partial \tilde{r}}.
\end{align*}
The set $Q \cap D_s$ is transformed in the new variables to the set $\Omega_s := \{(\tilde{r},\tilde{y}): \, \tilde{y} \in (y_1^Q,y_2^Q), \, \tilde{r} \in ( r_1/(4  g_s(\tilde{y})),1)\}$. Let us put $\psi_s = \psi_{D_s}$. Since $\psi_s$ satisfies (\ref{lapD}) in $D_s$ for $m = 1$, we obtain that, in the new coordinates,  $\psi_s$ satisfies in $\Omega_s$:
\begin{equation}
\label{opl}
\begin{aligned}
 & \frac{1  + (\tilde{r} s g'(\tilde{y}))^2}{(g_s(\tilde{y}))^2} \, \frac{\partial^2 \psi_s}{\partial \tilde{r}^2}
 + \frac{\partial^2 \psi_s}{\partial \tilde{y}^2}
 - \frac{2 \tilde{r} s g'(\tilde{y})}{g_s(\tilde{y})} \frac{\partial^2 \psi_s}{\partial \tilde{y} \partial \tilde{r}} \\
 & \qquad + \frac{1 + 2 (\tilde{r} s g'(\tilde{y}))^2 - \tilde{r}^2 s g''(\tilde{y}) g_s(\tilde{y})}{\tilde{r} (g_s(\tilde{y}))^2} \frac{\partial \psi_s}{\partial \tilde{r}} - \frac{1}{(\tilde{r} g_s(\tilde{y}))^2} \, \psi_s = 0 .
\end{aligned}
\end{equation}
Now we will  verify the assumptions of \cite[Lemma~6.29]{GT1977}.  Note that $0 <  r_1 \le g_s(\tilde{y}) \le 1$, and $0 \le \tilde{r} s g'(\tilde{y}) / g_s(\tilde{y}) \le C_1(D)$ on $\Omega_s$. Hence, by an elementary calculation, there exists a constant $\lambda  = \lambda(D) > 0$ such that for all $\xi, \eta \in \R$, $s \in [0,1]$, $(\tilde{r},\tilde{y}) \in \Omega_s$, we have $(\xi -  (\tilde{r} s g'(\tilde{y}) / g_s(\tilde{y})) \eta)^2 + (1 / (g_s(\tilde{y}))^2) \eta^2 \ge \lambda (\xi^2 + \eta^2)$. This means that for all $s \in [0,1]$,  the operator in~\eqref{opl} is strictly elliptic on $\Omega_s$ with a constant $\lambda$ which does not depend on $s$.

Note that the unit outer normal derivative $\frac{\partial}{\partial n}$ on $Q \cap D_s$ equals
\[
\frac{\partial}{\partial n} =  \frac{1}{(1 + (s g'(y))^2)^{1/2}} \left( \frac{\partial}{\partial r} - s g'(y) \frac{\partial}{\partial y} \right) .
\]
Since $\psi_s$ satisfies $\frac{\partial \psi_s}{\partial n} = 0$ on $Q \cap \partial D_s$ we obtain that in the new coordinates $\psi_s$ satisfies
\begin{equation}
\label{eq:boundary}
\frac{1 + (s g'(\tilde{y}))^2 \tilde{r}}{g_s(\tilde{y})} \frac{\partial \psi_s}{\partial \tilde{r}} (\tilde{r},\tilde{y}) - s g'(\tilde{y}) \frac{\partial \psi_s}{\partial \tilde{y}} (\tilde{r},\tilde{y}) = 0,
\end{equation}
for $\tilde{r} = 1$ and $\tilde{y} \in (y_1^Q,y_2^Q)$. Note that the coefficient at $\frac{\partial \psi_s}{\partial \tilde{r}}$ in the above formula is bounded from below by $\kappa = 1$, which does not depend on $s$.

We will use \cite[Lemma~6.29]{GT1977} for $\alpha = 1$, $u = \psi_s$, $\Omega = \Omega_s$, $T = \{(\tilde{r},\tilde{y}): \, \tilde{r} = 1, \, \tilde{y} \in (y_1^Q,y_2^Q)\}$ and $L$ the operator in \eqref{opl}. We have already checked that $\kappa > 0$ does not depend on $s$ and that the operator $L$ on $\Omega_s$ is strictly elliptic with a constant $\lambda$ not depending on $s$. It is also easy to check that absolute values of all the coefficients of $L$ on $\Omega_s$ and the coefficients in (\ref{eq:boundary}) on $T$ are bounded from above by a constant not depending on $s$. 

Note that $R \cap D_s$ is transformed in new variables to the set $U_s := \{(\tilde{r},\tilde{y}): \, \tilde{y} \in (y_1^R,y_2^R), \, \tilde{r} \in (r_1/(2  g_s(\tilde{y}) ),1)\}$. It is easy to check that
$$
\sup_{(r,y)\in \overline{D_s \cap R}} 
\sum_{\alpha + \beta = 2} \left|\frac{\partial^{\alpha + \beta}}{\partial^{\alpha} r \partial^{\beta} y} \psi_{s}(r,y) \right| 
\le
C(D) \sup_{(\tilde{r},\tilde{y})\in \overline{U_s}} 
\sum_{\alpha + \beta \le 2} \left|\frac{\partial^{\alpha + \beta}}{\partial^{\alpha} \tilde{r} \partial^{\beta} \tilde{y}} \psi_{s}(\tilde{r},\tilde{y}) \right|.
$$
Note that $\dist(U_s,\partial \Omega_s \setminus T)$ is bounded from below by a positive constant not depending on $s$. By \cite[Lemma~6.29]{GT1977} for $u$, $\Omega$, $T$ and $L$ as above we get 
$$
\sup_{(\tilde{r},\tilde{y})\in \overline{U_s}} 
\sum_{\alpha + \beta \le 2} \left|\frac{\partial^{\alpha + \beta}}{\partial^{\alpha} \tilde{r} \partial^{\beta} \tilde{y}} \psi_{s}(\tilde{r},\tilde{y}) \right| \le C(D) 
\sup_{(\tilde{r},\tilde{y})\in \overline{U_s}} |\psi_{s}(\tilde{r},\tilde{y})|.
$$
We have $\sup_{(\tilde{r},\tilde{y})\in \overline{U_s}} |\psi_{s}(\tilde{r},\tilde{y})| = \sup_{(r,y)\in \overline{D_s \cap R}} |\psi_{s}(r,y)|$. By Lemma \ref{lem:cylinder} this is bounded by a constant not depending on $s$. This implies (\ref{eq:uniformnorm}).

It is clear that for $\varepsilon = \min(r_1/2,T - T_3)$, $H = -y_0 + 1$ we have $W_s \in \calW(\varepsilon,1,H,1)$ for any $s \in [0,1]$.
The convergence $d(W_s,W_q) \to 0$ follows by Lemma \ref{lem:uniformapprox1}.
\end{proof}

\begin{lemma}
\label{lem:limit}
Let $D \in \calD_1$. Let $D_s$, $W_s$, $s \in [0,1]$, be  as in Lemma \ref{lem:uniformnorm} (see Figure \ref{fig:7}). Fix $q \in [0,1]$. Let $R$ be the rectangle defined in Lemma \ref{lem:uniformnorm} (see Figure \ref{fig:8}). Then we have
\[
\sup_{p \in \overline{D_s \cap D_{q} \cap R}} 
\left|\frac{\partial}{\partial y}(\psi_{D_s} - \psi_{D_{q}})(p)\right| \to 0
\quad \text{as} \quad
s \to  q.
\]
\end{lemma}

\begin{proof}
 Denote $\psi_s = \psi_{D_s}$. By Lemma~\ref{lem:uniformnorm}, there are $\ve, H$ such that $W_s \in \calW(\ve, 1, H, 1)$ for all $s \in [0, 1]$, and $d(W_s, W_q) \to 0$ as $s \to q$. Hence, by Lemma~\ref{lem:nablaW1W2},
\begin{align}
\label{eq:partialy0}
 & \int_{D_s \cap D_q} r \left|\pd[\psi_s]{y} - \pd[\psi_q]{y}\right|^2 \to 0  && \text{as $s \to q$.}
\end{align}
Again by Lemma~\ref{lem:uniformnorm}, $\pd[\psi_s]{y}$ is a Lipschitz function on $D_s \cap R$ with Lipschitz constant $a = a(D)$ not depending  on $s \in [0, 1]$. Let $\delta \in (0, 1)$. It follows that for any $p, p' \in D_s \cap D_q \cap R$ with $|p - p'| < \delta$ we have
\begin{align*}
 \left|\pd[\psi_s]{y}(p) - \pd[\psi_q]{y}(p)\right|^2 \le \left(\left|\pd[\psi_s]{y}(p') - \pd[\psi_q]{y}(p')\right| + 2 a \delta \right)^2 \le 8 a^2 \delta^2 + 2 \left|\pd[\psi_s]{y}(p') - \pd[\psi_q]{y}(p')\right|^2 .
\end{align*}
Furthermore, $|B(p, \delta) \cap D_s \cap D_q \cap R| \ge C(D) \delta^2$. Hence,
\begin{align*}
 & \left|\pd[\psi_s]{y}(p) - \pd[\psi_q]{y}(p)\right|^2 \le 8 a^2 \delta^2 + \frac{2}{C(D) \delta^2} \int_{B(p, \delta) \cap D_s \cap D_q \cap R} \left|\pd[\psi_s]{y} - \pd[\psi_q]{y}\right|^2 .
\end{align*}
Since $r \ge r_1 / 2$ on $R$, we conclude that
\begin{align*}
 & \sup_{p \in D_s \cap D_q \cap R} \left|\pd[\psi_s]{y}(p) - \pd[\psi_q]{y}(p)\right|^2 \le 8 a^2 \delta^2 + \frac{4}{C(D) \delta^2 r_1} \int_{D_s \cap D_q} r \left|\pd[\psi_s]{y} - \pd[\psi_q]{y}\right|^2 .
\end{align*}
By~\eqref{eq:partialy0}, it follows that
\begin{align*}
 \limsup_{s \to q} \left(\sup_{p \in D_s \cap D_q \cap R} \left|\pd[\psi_s]{y}(p) - \pd[\psi_q]{y}(p)\right|\right) \le 2 \sqrt{2} \, a \delta .
\end{align*}
Since $\delta \in (0, 1)$ was arbitrary, the proof is complete.
\end{proof}

\begin{lemma}
\label{lem:cylinder1}
Let $h > 0$ and $W = \{(x,y,z): \, x^2 + z^2 < 1, \, 0 > y  > -h\}$ be a cylinder. Then $\frac{\partial \vp}{\partial y} \ge 0$ on $W_+$. In particular $\frac{\partial \psi}{\partial y} \ge 0$ on $D = \{(x,y): \, x \in (0,1), y \in (-h,0)\}$. 
\end{lemma}

\begin{proof}
The result follows from explicit formulas for sloshing eigenfunctions in cylinders, we have $\psi_W(r,y) = c(h) J_1(j_{1,1}' r) \cosh(j_{1,1}' (y + h))$, where $J_1$ is the Bessel function of the first kind, and $j_{1,1}'$ is the first zero of its derivative (see e.g. \cite[page 502]{KK2009} or \cite[page 24]{BKPS2010}).
\end{proof}

The next result shows that the assertion of Theorem \ref{th:oddmon} holds for $W \in \calW_1$. The proof of this result uses some ideas from Proposition 3.2 \cite{JN2000}.

\begin{proposition}
\label{prop:main1}
Let $D \in \calD_1$ and $W = W(D) \in \calW_1$. Then we have $\frac{\partial \vp}{\partial x} \ge 0$ on $W$ and $\frac{\partial \vp}{\partial y} \ge 0$ on $W_+$.
\end{proposition}

\begin{proof}
In view of Lemmas \ref{lem:B2y} and \ref{lem:B2x} it is enough to show that $\frac{\partial \psi}{\partial y} \ge 0$ on $\overline{D}$. On the contrary assume that there exists a domain $D \in \calD_1$ and a point $p \in \overline{D}$ such that $\frac{\partial \psi}{\partial y}(p) < 0$.

Let $D_s$, $s \in [0,1]$ be the family of domains constructed before the statement of Lemma \ref{lem:uniformnorm} for the above domain $D = D_1$, and denote $\psi_s = \psi_{D_s}$, $\vp_s = \vp_{D_s}$. Let
\[
 q = \inf\left\{  s \in [0, 1] : \frac{\partial \psi_{s}}{\partial y}(p) < 0  \; \text{for some $p \in \overline{D_s}$} \right\}.
\]
We will first show that 
\begin{equation}
\label{eq:Dt0}
\frac{\partial \psi_{{q}}}{\partial y} \ge 0
\quad \text{on} \quad \overline{D_{q}}.
\end{equation}
If $q = 0$ this follows from Lemma \ref{lem:cylinder1} (here we use the fact that we know that the assertion of the proposition holds for cylinders). 

Now assume that $q > 0$. Then for all $s < q$ we have $\frac{\partial \psi_{s}}{\partial y} \ge 0$ on $D_s$. Note that for all $s < q$ we have $D_{q} \subset D_s$. Let $R$ be the rectangle defined in Lemma \ref{lem:uniformnorm}. By Lemma \ref{lem:limit} (for $s < q$) we have
$$
\sup_{p \in \overline{D_s \cap D_{q} \cap R}} 
\left|\frac{\partial}{\partial y}(\psi_{s} - \psi_{{q}})(p)\right| 
= \sup_{p \in \overline{D_{q} \cap R}} 
\left|\frac{\partial}{\partial y}(\psi_{s} - \psi_{{q}})(p)\right|
\to 0,
$$
as $s \to q^-$. This with the fact that for all $s < q$ we have $\frac{\partial \psi_{s}}{\partial y} \ge 0$ on $D_s$ implies that $\frac{\partial \psi_{{q}}}{\partial y} \ge 0$ on $\overline{D_{q} \cap R}$. But $B_2(D_{q}) \subset \overline{D_{q} \cap R}$ so Lemma \ref{lem:B2y} implies that $\frac{\partial \psi_{{q}}}{\partial y} \ge 0$ on $\overline{D_{q}}$.

We have shown (\ref{eq:Dt0}). Note that by Lemma \ref{lem:B2x} we get $\frac{\partial \vp_{q}}{\partial x} \ge 0$ on $\overline{W(D_q)}$. It follows that $\frac{\partial \psi_{{q}}}{\partial r} \ge 0$ on $\overline{D_{q}}$. Now we will show that there exist a point $p_0 \in \overline{B_2(D_{q})}$ such that
$$
\frac{\partial \psi_{{q}}}{\partial y}(p_0) = 0.
$$

By (\ref{eq:Dt0}) we get that $q < 1$. By the definition of $q$ we get that there exists a decreasing sequence $\{s_n\}_{n = 1}^{\infty} \subset [0,1]$, $ s_n \to q$, such that for all $n \in \N$  we have $(\partial \psi_{{s_n}} / \partial y) (p_n) < 0$  for some $p_n \in \overline{D_{s_n}}$.

By Lemma \ref{lem:B2y} we may assume that $p_n \in B_2(D_{s_n})$. Note that $D_{s_n} \subset D_{q}$. We also may assume (after taking a subsequence if necessary) that $p_n \to p_0 \in \overline{B_2(D_{q})}$. Note also that $p_n \in \overline{R}$ for any $n = 0, 1, 2, \ldots$. We have
\begin{equation}
\label{eq:DtnDt0}
\begin{aligned}
 \left|\frac{\partial \psi_{{s_n}}}{\partial y}(p_n)
- \frac{\partial \psi_{{q}}}{\partial y}(p_0)\right| 
& \le \left|\frac{\partial \psi_{{s_n}}}{\partial y}(p_n)
- \frac{\partial \psi_{{q}}}{\partial y}(p_n)\right|
+ \left|\frac{\partial \psi_{{q}}}{\partial y}(p_n)
- \frac{\partial \psi_{{q}}}{\partial y}(p_0)\right| \\
& \hspace{-2em} \le \sup_{p \in \overline{D_{s_n} \cap R}} 
\left|\frac{\partial \psi_{{s_n}}}{\partial y}(p)
- \frac{\partial \psi_{{q}}}{\partial y}(p)\right|
+ \left|\frac{\partial \psi_{{q}}}{\partial y}(p_n)
- \frac{\partial \psi_{{q}}}{\partial y}(p_0)\right|.
\end{aligned}
\end{equation}
By Lemma \ref{lem:limit} the first expression in (\ref{eq:DtnDt0}) tends to $0$ as $n \to \infty$. The second expression in (\ref{eq:DtnDt0}) tends to $0$ as $n \to \infty$ because $p_n \to p_0$, $p_n \in \overline{D_{q} \cap R}$ ($n = 0, 1, 2, \ldots$) and $\frac{\partial \psi_{{q}}}{\partial y}$ is a Lipschitz function on $\overline{D_{q} \cap R}$ by Lemma \ref{lem:uniformnorm}. 

Since $\frac{\partial \psi_{{s_n}}}{\partial y}(p_n) < 0$ it follows that $\frac{\partial \psi_{{q}}}{\partial y}(p_0) \le 0$. We know that $\frac{\partial \psi_{{q}}}{\partial y} \ge 0$ and $\frac{\partial \psi_{{q}}}{\partial r} \ge 0$ on $\overline{D_{q}}$, so $\frac{\partial \psi_{{q}}}{\partial y}(p_0) = 0$. But $p_0 \in \overline{B_2(D_{q})}$, a contradiction with Lemma \ref{lem:yderivative}.
\end{proof}

\section{Monotonicity of the odd eigenfunction}
In the previous section we proved monotonicity properties of $\psi_{1,1}$ for a special class of piecewise smooth domains $\calW_1$. In this section we pass to the limit to obtain the same result for class $\calW$. In the whole section we use notation \eqref{psi11}.

\begin{proof}[proof of Theorem \ref{th:oddmon}]
Note that by scaling it is sufficient to show the assertion of the theorem for $D \in \calD$ such that $r_0 = 1$. 

Let us consider the following inequality:
\begin{align}
\label{eq:in}
 \psi(r_2,y_2) & \ge \psi(r_1,y_1) && \text{for any $(r_1,y_1), (r_2,y_2) \in D$, $r_2 \ge r_1$, $y_2 \ge y_1$.}
\end{align}
Our first aim will be to justify  \eqref{eq:in} for all $D \in \calD$ such that $r_0 = 1$. To show this inequality we will use the following scheme. If we have a family of sets $\{D_s: \, s \in [0,s_0]\}$ such that all these sets belong to $\calD(\eps,M,H,1)$, $W_s = W(D_s)$, $d_{\calW(\eps,M,H,1)}(W_s,W_0) \to 0$ as $s \to 0$ and  \eqref{eq:in} holds for all $D_s$ where $s \in (0,s_0]$ then Lemma \ref{familyWt} guarantees that  \eqref{eq:in} holds also for $D_0$. We will show that  \eqref{eq:in} holds for all $D \in \calD$ such that $r_0 = 1$ in a number of steps.

{\bf{Step 1.}} Note that Proposition \ref{prop:main1} gives  \eqref{eq:in} for $D \in \calD_1$.

{\bf{Step 2.}} Let $\calD_2$ be the class of domains satisfying all conditions for the class $\calD_1$ except that the condition (e) is replaced by
\begin{enumerate}
\item[(e')] $t \to r(t)$ is a continuous function on $[T_1,T]$.
\end{enumerate}
We will show that  \eqref{eq:in} holds for $D \in \calD_2$. Fix  $D \in \calD_2$ and let $y_0$, $ r_1  = T_1$, $T_2$, $T_3$, $T$ such as in Definition \ref{classD1}. Let $\eps = \min(r_1,T - T_3,T_2 - T_1)$, $H = -y_0 + 1$,  so that $D \in \calD(\ve, 1, H, 1)$, and let $r(t)$, $y(t)$, $t \in [0, T]$, be the parametrization of $B(D)$. 

For $s > 0$ let $h_s \in C^{\infty}(\R)$ be such that $h_s \ge 0$, $\supp(h_s) \subset (-s,s)$, $\int_{\R} h_s = 1$. Put $s_0 = \eps/2$. For $s \in (0,s_0]$ let $y_s(t) = y(t)$  for $t \in [0,T]$, $r_s(t) = r(t)$ for $t \in [0,T_1 + \eps/2] \cup [T - \eps/2,T]$ and $r_s(t) = \int_{-s}^s r(t - u) h_s(u) \, du$ for $t \in [T_1 + \eps/2, T - \eps/2]$. Let $D_0 = D$ and  for $s \in (0,s_0]$, let $D_s$ be a domain defined like $D$ but with $(r(t),y(t))$ replaced by $(r_s(t),y_s(t))$. Note that $D_s \in \calD_1$  for $s \in (0,s_0]$ and $D_s \in \calD(\eps/2,1,H,1)$  for $s \in [0,s_0]$. We have $\|y_s - y_0\|_{\infty} = 0$ and $\|r_s - r_0\|_{\infty} \to 0$ as $s \to 0$ so Lemma \ref{lem:uniformapprox1} gives that $d_{\calW(\eps/2,1,H,1)}(W_s,W_0) \to 0$ as $s \to 0$, where $W_s = W(D_s)$. Hence Step 1 and Lemma \ref{familyWt} give that  \eqref{eq:in} holds for  $D_0 = D \in \calD_2$.

{\bf{Step 3.}} Let $\calD_3$ be the class of domains satisfying all conditions for the class $\calD_1$ except that condition (e) is deleted and condition (c) is replaced by
\begin{enumerate}
\item[(c')] $t \to y(t)$, $t \to r(t)$ are strictly increasing continuous functions on $[T_2,T_3]$.
\end{enumerate}
It is clear that $\calD_2 \subset \calD_3$, and since we can replace $(r(t), y(t))$ by $(r(y^{-1}(y_0 + t - T_1)), y_0 + t - T_1)$ for $t \in [T_2, T_3]$ (this is just a reparametrization), we see that in fact $\calD_3 = \calD_2$.

{\bf{Step 4.}} Let $\calD_4$ be the class of domains satisfying all conditions for the class $\calD_1$ except that condition (e) is deleted and condition (c) is replaced by
\begin{enumerate}
\item[(c'')] $t \to y(t)$, $t \to r(t)$ are nondecreasing continuous functions on $[T_2,T_3]$.
\end{enumerate}

Fix $D \in \calD_4$ and let $y_0$, $r_1 = T_1$, $T_2$, $T_3$, $T$ be such as in Definition \ref{classD1}. Let $\eps = \min(r_1,T - T_3)$, $H = -y_0 + 1$,  so that $D \in \calD(\ve, 1, H, 1)$, and let $r(t)$, $y(t)$, $t \in [0, T]$, be the parametrization of $B(D)$.

For $s \in [0, 1)$ define $r_s(t) = r(t)$ and $y_s(t) = y(t)$ for $t \in [0, T_2] \cup [T_3, T]$, and let
\begin{align*}
 r_s(t) & = (1 - s) r(t) + s \! \left(r_1 + \frac{t - T_2}{T_3 - T_2} \, (1 - r_1)\right) , \\ y_s(t) & = (1 - s) y(t) + s (y_0 + t - T_1)
\end{align*}
for $t \in [T_2, T_3]$. Let $D_s$ be the corresponding domain. Note that $D_s \in \calD_3$  for $s \in (0,s_0]$ and $D_s \in \calD(\eps,1,H,1)$  for $s \in [0,s_0]$. We have $\|y_s - y_0\|_{\infty} \to 0$ and $\|r_s - r_0\|_{\infty} \to 0$ as $s \to 0$ so Lemma \ref{lem:uniformapprox1} gives that $d_{\calW(\eps,1,H,1)}(W_s,W_0) \to 0$ as $s \to 0$, where $W_s = W(D_s)$. Hence Step 2, Step 3 and Lemma \ref{familyWt} give that  \eqref{eq:in} holds for  $D_0 = D \in \calD_4$.

{\bf{Step 5.}} Let $\calD_5$ be the class of domains satisfying all conditions for the class $\calD_1$ except that conditions (b) and (e) are deleted, $T_2 = T_1$ and condition (c) is replaced by
\begin{enumerate}
\item[(c'\kern-0.08em '\kern-0.08em ')] $t \to y(t)$, $t \to r(t)$ are nondecreasing continuous functions on $[T_1,T_3]$.
\end{enumerate}
Fix $D \in \calD_5$ and let $y_0$, $r_1 = T_1 = T_2$, $T_3$, $T$ be such as in Definition \ref{classD1}. Let $\eps = \min(r_1,T - T_3)$, $H = -y_0 + 1$,  so that $D \in \calD(\ve, 1, H, 1)$, and let $r(t)$, $y(t)$, $t \in [0, T]$, be the parametrization of $B(D)$.

Let $s \in [0, 1/2]$ and $T_2(s) = (1 - s) T_1 + s T_3$. We define $r_s(t) = r(t)$ and $y_s(t) = y(t)$ for $t \in [0, T_1] \cup [T_3, T]$, $r_s(t) = r_1$ and $y_s(t) = y_0 + t - T_1$ for $t \in [T_1, T_2(s)]$, and let
\begin{align*}
 r_s(t) & = r \! \left( T_3 - \frac{T_3 - t}{1 - s} \right) , & y_s(t) & = (1 - s) y \! \left( T_3 - \frac{T_3 - t}{1 - s} \right) + s y(T_3)
\end{align*}
for $t \in [T_2(s), T_3]$. The corresponding domains $D_s$ are in $\calD_4$ for $s \in (0,1/2]$, and $D_s \in \calD(\eps,1,H,1)$ for $s \in [0,1/2]$. Since $r(t)$ and $y(t)$ are continuous, we have $\|y_s - y_0\|_{\infty} \to 0$ and $\|r_s - r_0\|_{\infty} \to 0$ as $s \to 0$. Again, Lemma~\ref{lem:uniformapprox1} and Lemma~\ref{familyWt} together with Step~4 give that~\eqref{eq:in} holds for  $D_0 = D \in \calD_5$.

{\bf{Step 6.}} Let  $\calD_6$ be the class of all domains from the class $\calD$ such that $r_0 = 1$. Fix  $D \in \calD_6$ and let $\eps$, $M$, $H$, $T$, $y_0$, $r(t)$, $y(t)$ be such as in Definition~\ref{classD}.  Let $s \in [0, \eps/2]$. We define $r(t) = t$ and $y(t) = y_0$ for $t \in [0, s]$, $r(t) = 1$ and $y(t) = t - T$ for $t \in [T - s, T]$, and
\begin{align*}
 r_s(t) & = s + (1 - s) r \! \left(\frac{t - s}{T - 2 s} \, T \right) , & y_s(t) & = -s + \frac{-y_0 - s}{-y_0} \, y \! \left(\frac{t - s}{T - 2s} \, T \right)
\end{align*}
for $t \in [s, T - s]$. Let $D_s$ correspond to $r_s(t)$ and $y_s(t)$ for $s \in [0, \eps/2]$. Observe that $D_s \in \calD(\eps, 2M, H, 1)$ for $s \in [0, \eps/2]$ and $D_s  \in \calD_5$ for $s \in (0,\eps/2]$. Furthermore, $\|y_s - y_0\|_{\infty} \to 0$ and $\|r_s - r_0\|_{\infty} \to 0$ as $s \to 0$, so Lemma~\ref{lem:uniformapprox1}, Lemma~\ref{familyWt} and Step~5 give~\eqref{eq:in} for  $D_0 = D \in \calD_6$.

{\bf{Step 7.}} Step 6 shows  \eqref{eq:in} for all $D \in \calD$ such that $r_0 = 1$. This implies the assertion of Theorem \ref{th:oddmon} but with weak inequalities instead of strict ones. By Lemma \ref{lem:global} this gives $\frac{\partial \varphi}{\partial x} \ge 0$ on $W$ and $\frac{\partial \varphi}{\partial y} \ge 0$ on $W_+$. However $\frac{\partial \varphi}{\partial x}$, $\frac{\partial \varphi}{\partial y}$ are harmonic functions in $W$ which implies $\frac{\partial \varphi}{\partial x} > 0$ on $W$ and $\frac{\partial \varphi}{\partial y} > 0$ on $W_+$. Of course, this gives strict inequalities in Theorem \ref{th:oddmon}.
\end{proof}

\begin{proof}[proof of Theorem \ref{th:mainconvex} (ii), (iii)]
Elementary geometric considerations (we omit the details here) give that any axisymmetric, convex, bounded liquid domain satisfying the John condition belongs to the class of domains $\calD$. Hence Theorem \ref{th:mainconvex} (ii) follows from Theorem \ref{th:oddmon}. The inequalities in Theorem \ref{th:mainconvex} (iii) follows from Theorem \ref{th:mainconvex} (ii) by Lemma \ref{lem:global}. The last sentence in Theorem \ref{th:mainconvex} (iii) is an easy corollary of Theorem \ref{th:mainconvex} (i).
\end{proof}

\begin{proof}[proof of Proposition \ref{notJohn}]
Let $F = \{(x,y,z): \, x^2 + z^2 < R^2, y = 0\}$ for some $R > 0$. Let $\theta_0 \in (\pi/2,\pi)$ denote the angle between the free surface $F(D)$ and the rigid wall $B(D)$. It was proved in~\cite{LBK1984} (formula~(13.3)) that (see also~\cite{K1980})
\begin{align*}
 \psi(r, 0) & = \psi(R, 0) (1 - \nu (R - r) \cot \theta_0 ) + o(R - r) && \text{as $r \to R^-$.}
\end{align*}

Let us recall that we may assume that $\psi > 0$ on $D$. It is a standard result that $\psi$ is continuous on $\overline{D}$ and $\psi(0,0) = 0$. Of course $\vp$ attains maximum on $\overline{F}$ so $\psi$ must attain maximum on $\overline{F(D)}$. 

If $\psi(R,0) = 0$ then clearly $\psi(r,0)$ does not attain its maximum at $r = R$. If $\psi(R,0) > 0$ then $\frac{\partial \psi}{\partial r}(R, 0) = \nu \cot \theta_0 \psi(R, 0) < 0$ and therefore $\psi(R,0)$ cannot be the maximum of $\psi(r,0)$. 

It follows that $\psi(r,0)$ attains its maximum inside the interval $(0,R)$ and $\frac{\partial \psi}{\partial r}(r,0)$ changes the sign on $(0,R)$. Hence $\psi(r,y)$ has its maximum in the interior of $F(D)$ and $\frac{\partial \psi}{\partial r}(r,y)$ changes the sign in $D$.
\end{proof}

\vskip 10 pt

\noindent \textbf{Acknowledgments.} The authors would like to thank N. Kuznetsov for many useful discussions on the subject of the paper during his two visits in Wroclaw and the authors' visit in St. Petersburg. The authors would also like to thank N. Kuznetsov for sharing his knowledge about the sloshing problem and his help in preparing the introduction section in this paper.

\end{document}